%% file: tvwavecontrol.tex
\pdfoutput=1
\documentclass[english]{jnsao}
\usepackage[utf8]{inputenc}

\newcommand{\eps}{\varepsilon}
\renewcommand{\phi}{\varphi}
\newcommand{\N}{\mathbb{N}}
\newcommand{\R}{\mathbb{R}}
\newcommand{\LL}{\mathbb{L}}
\newcommand{\Rbar}{\overline{\R}}

\newcommand{\calF}{\mathcal{F}}

\newcommand{\calM}{\mathcal{M}}
\newcommand{\calO}{\mathcal{O}}

\newcommand{\calR}{\mathcal{R}}
\newcommand{\calT}{\mathcal{T}}
\newcommand{\dx}{\,\mathrm{d}x}
\newcommand{\dt}{\,\mathrm{d}t}
\newcommand{\ds}{\,\mathrm{d}s}
\newcommand{\TV}{\mathrm{TV}}
\newcommand{\dual}[1]{\langle #1 \rangle}
\newcommand{\scalprod}[1]{\left( #1 \right)}
\newcommand{\norm}[1]{\| #1 \|}
\newcommand{\set}[2]{\left\{#1:#2\right\}}
\newcommand{\wkto}{\rightharpoonup}

\DeclareMathOperator{\Id}{\mathrm{Id}}

\renewcommand{\div}{\operatorname{\mathrm{div}}}
\newcommand{\prox}{\mathrm{prox}}
\newcommand{\proj}{\mathrm{proj}}

\usepackage[centercolon]{mathtools}

\usepackage{booktabs}
\usepackage{graphicx}
\usepackage{subcaption}
\usepackage{tikz}
\usepackage{pgfplots}
\pgfplotsset{compat=newest}
\pgfplotsset{plot coordinates/math parser=false}
\pgfplotsset{every axis plot/.append style={line width=0.5pt}, tick label style={font=\small}, yticklabels={,,}, legend pos=south east}

\usepackage[capitalise,nameinlink]{cleveref}
\numberwithin{equation}{section}
\newtheorem{assumption}{Assumption}
\crefname{assumption}{Assumption}{Assumptions}
\usepackage{enumitem}
\setlist[enumerate]{label={(\roman*)}}

\def\thetitle{Optimal control of the principal coefficient in a scalar wave equation}
\shortauthor{Clason, Kunisch, Trautmann}

\manuscriptstatus{Submitted manuscript}
\manuscripteprinttype{arxiv}
\manuscripteprint{1912.08672v3}
\date{2020-11-12}
\manuscriptlicense{CC-BY}

\title{\thetitle}
\author{%
  Christian Clason\thanks{Faculty of Mathematics, University Duisburg-Essen, 45117 Essen, Germany (\email{christian.clason@uni-due.de}, \orcid{0000-0002-9948-8426})}
  \and
  Karl Kunisch\thanks{Institute of Mathematics and Scientific Computing, University of Graz, Heinrichstrasse 36, 8010 Graz, Austria, and Radon Institute, Austrian Academy of Sciences, Linz, Austria (\email{karl.kunisch@uni-graz.at})} 
  \and 
  Philip Trautmann\thanks{Institute of Mathematics and Scientific Computing, University of Graz, Heinrichstrasse 36, 8010 Graz, Austria (\email{philip.trautmann@uni-graz.at})} 
}

\AtBeginDocument{
  \hypersetup{
    pdftitle = {\thetitle},
    pdfauthor = {C.~Clason and K.~Kunisch and P.~Trautmann}
  }
}

\begin{document}

\maketitle

\begin{abstract}
  We consider optimal control of the scalar wave equation where the control enters as a coefficient in the principal part. Adding a total variation penalty allows showing existence of optimal controls, which requires continuity results for the coefficient-to-solution mapping for discontinuous coefficients. We additionally consider a so-called \emph{multi-bang} penalty that promotes controls taking on values pointwise almost everywhere from a specified discrete set.
  Under additional assumptions on the data, we derive an improved regularity result for the state, leading to optimality conditions that can be interpreted in an appropriate pointwise fashion. The numerical solution makes use of a stabilized finite element method and a nonlinear primal-dual proximal splitting algorithm.
\end{abstract}

\section{Introduction}\label{sec:introduction}

This work is concerned with an optimal control problem for the scalar
wave equation where the control enters as the spatially varying coefficient
in the principal part. Informally, we consider the problem
\begin{equation}\label{eq:problem-intro}
  \left\{
    \begin{aligned}
      &\min_{u,y}\frac12\|By-y_d\|_{\calO}^2+\calR(u)\\
      &\text{s.t.}\quad y_{tt} - \div(u\nabla y) = f,\quad y(0)=y_0, \partial_t y(0) = y_1,\\
      &\phantom{s.t.}\quad \underline u \leq u \leq \overline u \quad\text{almost everywhere (a.e.)},
    \end{aligned}
  \right.
\end{equation}
where $y_d$ is a given (desired or observed) state, $B$ is a bounded linear
observation operator mapping to the observation space $\calO$, $\cal R$ is a
regularization term, $0<\underline u < \overline u$ are constants, and $f$,
$y_0$, and $y_1$ (as well as boundary conditions) are given suitably. A precise
statement is deferred to \cref{sec:preliminaries}. Such problems
occur, e.g., in acoustic tomography for medical imaging \cite{BC:2006a} and non-destructive testing \cite{Krautkramer:1990} as well as in seismic inversion \cite{Tarantola:1984}. In the latter, the goal is the
determination of a ``velocity model'' (as described by the coefficient $u$) of
the underground in a region of interest from recordings (``seismograms'',
modeled by $y_d$) of reflected pressure waves generated by sources on or near
the surface (entering the equation via $f$, $y_0$, $y_1$, or inhomogeneous
boundary conditions). If the region contains multiple different materials like
rock, oil, and gas, the velocity model changes rapidly or may even have jumps
between material interfaces.

In the stationary case, the question of existence of solutions to problem
\eqref{eq:problem-intro} under only pointwise constraints and regularization
has received a tremendous amount of attention. However, it was answered in the
negative in \cite{Murat:1977}; this and subsequent investigations led to the
concept of $H$-convergence and, more generally, to homogenization theory;
see, e.g., \cite{Jiang2005,Murat1997,Tartar:1987,Tartar97,Tartar:2009}. The
use of regularization terms or constraints involving higher-order differential
operators would certainly guarantee existence but contradicts the
goal of allowing piecewise continuous controls $u$. Such considerations
suggest the introduction of total variation regularization in addition to
pointwise constraints. In this case, existence can be argued. However, this
leads to difficulties in deriving necessary optimality conditions since the sum
rule of convex analysis can only be applied in the $L^\infty(\Omega)$ topology,
which would lead to (generalized) derivatives that do not admit a pointwise
representation. This difficulty can be circumvented by replacing the pointwise
constraints by a (differentiable approximation of a) cutoff function applied
to the coefficient in the equation and by using improved regularity results for
the optimal state that allow extending the Fréchet derivative of the tracking
term from $L^\infty(\Omega)$ to $L^s(\Omega)$ for $s<\infty$ sufficiently
large. Together, this allows obtaining derivatives and subgradients in
$L^r(\Omega)$ for some $r>1$, which can be characterized in a suitable pointwise
manner. This was carried out in \cite{CKK:2017}, which considered for $\calR$
a combination of total variation and multi-bang regularization; the latter
is a convex pointwise penalty that promotes controls which take values from
a prescribed discrete set (e.g., corresponding to different materials such as
rock, oil, and gas); see also \cite{CK:2013,CK:2015,CD:2017}.

In the current work, we extend this approach to optimal control and
identification of discontinuous coefficients in scalar wave equations by
deriving under additional (natural) assumptions on the data the adapted higher
regularity results for the wave equation based on elliptic maximal regularity
theory \cite{Groeger}; see \cref{ass:data,prop:state_regularity} below. We also
address a suitable discretization of the problem using a stabilized finite element method \cite{Zlotnik1994} and its solution by a nonlinear primal-dual proximal
splitting method \cite{Valkonen:2014,CV:2017,CMV:2019}.

Let us briefly comment on related literature. As there is a vast body of work
on control and inverse problems for the wave equation, we focus here
specifically on the identification of discontinuous (and, in particular,
piecewise constant) coefficients. This problem has attracted strong
interest over the last few decades, mainly due to its relevance in seismic
inversion. Classical works are mainly concerned with the one-dimensional
setting -- as a model for seismic inversion in stratified or layered media --
which allows making use of integral transforms to derive explicit
``layer-stripping'' formulas; see, e.g.,
\cite{Bube:1985,Lavrentev:1992,Aktosun:1996,Sedipkov:2014}. Regarding the
numerous works on wave speed identification in the multidimensional wave
equation for seismic inversion, we only mention exemplarily
\cite{Stolk00,Boehm,Goncharsky2019}; see also further literature cited there.
The use of total variation penalties for recovering a piecewise constant wave
speed in multiple dimensions has been proposed in, e.g., \cite{Akcelik:2008,
Burstedde09, Esser18, Yong:2018, Gao:2019}, although the earlier works employed
a smooth approximation of the total variation to allow the numerical solution
by standard approaches for nonlinear PDE-constrained optimization. Finally,
joint multi-bang and total variation regularization of linear inverse problems
and its numerical solution by a primal-dual proximal splitting methods were
considered in \cite{TramDo}. We also mention that multi-bang control is related to (but different from) switching controls, where at each instant in time, one and only one from a given set of time-dependent controls should be active; see, e.g., \cite{Hante:2013}.

This work is organized as follows. In the next \cref{sec:preliminaries}, we
give a formal statement of the optimal control problem \eqref{eq:problem-intro}
and recall the relevant definitions and properties of the functional. We then
derive in \cref{sec:state} the results on regularity, stability, and a priori estimates
for solutions of the state equation that will be needed in the rest of the
paper. In particular, in \cref{prop:state_regularity} we show a Groeger-type
maximal regularity result for the wave equation under additional assumptions on
the data. \Cref{sec:existence} is devoted to existence and first-order
necessary optimality conditions for optimal controls, where
we use the mentioned maximal regularity result to show that the latter can be
interpreted in a pointwise fashion. We then discuss the numerical computation
of solutions using a stabilized finite element discretization (see
\cref{sec:discretization}) together with a nonlinear primal-dual proximal
splitting method (see \cref{sec:pdps}). This approach is illustrated in
\cref{sec:examples} for two examples: a transmission setup motivated by
acoustic tomography and a reflection setup modeling seismic tomography.

\section{Problem statement}\label{sec:preliminaries}

Let $\Omega \subset \R^d$, $d \in \{2,3\}$, be a bounded domain
with $C^{2,1}$ regular boundary $\partial \Omega$ and outer normal $\nu$. For brevity, we introduce the notation $H:=L^2(\Omega)$ and $V:=H^1(\Omega)$ and set $I:=(0,T)$.
Then we consider for $f\in L^2(I,V)$, $y_0\in V$, and $y_1\in H$
the weak solution $y\in C(\overline
I,V)\cap C^1(\overline I,H)$ to
\begin{equation} \label{eq:wave-equation}
  \left\{\begin{aligned}
      \partial_{tt} y - \div(u\nabla y) &= f && \text{ in } Q
      := (0,T) \times \Omega, \\
      {\partial_\nu y} &= 0 && \text{ on } \Sigma := (0,T)
      \times \partial \Omega, \\
      y(0) &= y_0,\quad {\partial_t} y(0) = y_1, && \text{ on
      }\Omega.
  \end{aligned}\right.
\end{equation}
This choice of Neumann boundary conditions corresponds, e.g., for acoustic waves to the situation of reflection at a sound-hard obstacle and for elastic waves to the absence of external forces at the boundary (which is a natural setting for seismic imaging via interior sources).
We will discuss existence and regularity of solutions to \eqref{eq:wave-equation} in the following \cref{sec:state}.

The salient point is of course the coefficient $u$ in the principal part,
which we want to control on an open subset $\omega_c\subseteq \Omega$, which is assumed
to have a $C^{2,1}$ regular boundary.
For constants $\underline{u}, \overline{u}$ with $0 < \underline{u} <
\overline{u} < \infty$ we define the set of admissible coefficients
\begin{equation}\label{eq:def-hat-U}
  \hat U = \set{u \in L^\infty(\Omega)}{\underline{u} \le u(x) \le
  \overline{u}\quad\text{for a.e.}~x\in \Omega}
\end{equation}
and pick a reference coefficient $\hat u\in \hat U$.
To map a control $u$ defined on $\omega_c$ to a coefficient defined on
$\Omega$, we introduce the affine bounded extension operator
\begin{equation}\label{eq:extension}
  \hat E: L^2(\omega_c)\to L^2(\Omega),\qquad [\hat Eu](x) :=
  \begin{cases}
    \hat u(x) + u(x) & \text{for }x\in \omega_c,\\
    \hat u(x) & \text{for }x\in \Omega\setminus\omega_c.
  \end{cases}
\end{equation}
The set of controls that can be extended to admissible coefficients is
then given by
\begin{equation}\label{eq:def-U}
  U = \set{u \in L^\infty(\omega_c)}{u_{\min} \le u(x) \le u_{\max}
  \quad\text{for a.e.}~x\in \omega_c},
\end{equation}
where $u_{\min}<u_{\max}$ are such that $\underline{u} \leq \inf_{x\in \omega_c} \hat u(x) + u_{\min} \leq \sup_{x\in \omega_c} \hat u(x) + u_{\max}
\leq \overline{u}$. In particular, for $\hat u \equiv \underline u$, we have $u_{\min}=0$ and $u_{\max} = \overline{u}-\underline{u}$.

Moreover, we introduce the observation space $\calO$ which is assumed to be a
separable Hilbert space as well as a linear and bounded observation operator
$B\in \LL(L^2(Q),\calO)$ with adjoint $B^\ast\in \LL(\calO, L^2(Q))$.

We then consider the optimal control problem
\begin{equation}\label{eq:problem-full}
  \min_{u\in BV(\omega_c)\cap U}\frac12\|By(\hat
  Eu)-y_d\|_{\calO}^2+\alpha G(u)+\beta \TV(u),
\end{equation}
where $y(u)$ is a weak solution to \eqref{eq:wave-equation}, $G$
is the multi-bang penalty from \cite{CK:2013,CK:2015}, $\TV$ denotes
the total variation, and
$\alpha$ and $\beta$ are positive constants. In the remainder of
this section, we recall the definitions and properties of
the total variation and the multi-bang penalty relevant to the current work.

\paragraph{Total variation}

We recall, e.g., from
\cite{Ambrosio,Giusti,Ziemer} that the space $BV(\omega_c)$ is given by
those functions $v\in L^1(\omega_c)$ for which the distributional
derivative $Dv$ is a Radon measure, i.e.,
\begin{equation*}
  BV(\omega_c) = \set{v\in L^1(\omega_c)}{
  \norm{Dv}_{\calM(\omega_c)} < \infty}.
\end{equation*}
The \emph{total variation} of a function $v\in BV(\omega_c)$ is then
given by
\begin{equation*}
  \TV(v) := \norm{Dv}_{\calM(\omega_c)} = \int_{\omega_c}
  \mathrm{d}|Dv|_2,
\end{equation*}
i.e., the total variation (in the sense of measure theory) of the vector
measure $Dv\in\calM(\omega_c;\R^d)=C_0(\omega_c;\R^d)^*$. Here,
$|\cdot|_2$ denotes the Euclidean norm on $\R^d$; we thus consider in this work
the \emph{isotropic} total variation. For $v\in L^1(\omega_c)\setminus
BV(\omega_c)$, we set $\TV(v)=\infty$.
It follows that $BV(\omega_c)$ embeds into $L^r(\omega_c)$ continuously
for every $r\in [1,\frac{d}{d-1}]$ and compactly if $r < \frac{d}{d-1}$;
see, e.g., \cite[Cor.~3.49 together with Prop.~3.21]{Ambrosio}.
In addition, the total variation is lower semi-continuous with respect
to strong convergence in $L^1(\omega_c)$, i.e., if
$\{u_n\}_{n\in\N}\subset BV(\omega_c)$ and $u_n\to u$ in
$L^1(\omega_c)$, we have that
\begin{equation}\label{eq:tv_lsc}
  \TV(u)\leq \liminf_{n\to\infty} \TV(u_n),
\end{equation}
see, e.g., \cite[Thm.~5.2.1]{Ziemer}. Note that this does not imply that
$\TV(u)<\infty$ and hence that $u\in BV(\omega_c)$ unless
$\{\TV(u_n)\}_{n\in\N}$ has a bounded subsequence.
From \eqref{eq:tv_lsc}, we also deduce that the convex extended
real-valued functional $\TV:L^p(\omega_c)\rightarrow\R\cup\{\infty\}$ is
weakly lower semi-continuous for any $p\in [1,\infty]$.

\paragraph{Multi-bang penalty}

Let $u_{\min} \leq u_1<\dots < u_m\leq u_{\max}$ be a given set of desired
coefficient values.
The \emph{multi-bang penalty} $G$ is then defined similar to \cite{CK:2015},
where we have to replace the box constraints $u(x)\in [u_1,u_m]$ by a
linear growth to ensure that $G$ is finite on $L^r(\omega_c)$,
$r<\infty$. For simplicity, we assume in the following that $u_1 = u_{\min} = 0$
and $u_m = u_{\max} = \overline u - \underline u$ (i.e., $\hat u = \underline u$) and define
\begin{equation*}
  G:L^1(\omega_c)\to\R,\qquad G(u) = \int_{\omega_c}
  g(u(x))\,\mathrm{d}x,
\end{equation*}
where $g:\R\to \R$ is given by
\begin{equation}\label{eq:multi-bang_pw}
  g(t) = \begin{cases}
    -u_m t & t \leq u_1,\\
    \frac12 \left((u_{i}+u_{i+1})t - u_iu_{i+1}\right) & t \in
    [u_i,u_{i+1}],\quad 1\leq i <m,\\
    u_m t - \frac12 u_m^2 & t \geq u_m.
  \end{cases}
\end{equation}
This definition can be motivated via the convex envelope of $\delta_{[u_1,u_d]}(t) + \frac12 |t|^2$ (where $\delta$ denotes the indicator function in the sense of convex analysis), see \cite{CK:2015}; note however that here (as in \cite{CKK:2017}) $g$ is defined to be finite for every $t\in\R$, while the convex envelope is only finite for $t\in[u_1,u_2]$.
We also remark that for $m=2$, this reduces in the current setting to the
well-known \emph{sparsity penalty} (i.e., $G(u) = \norm{u}_{L^1(\omega_c)}$ for any $u\in U$).

It can be verified easily that $g$ is continuous, convex, and linearly
bounded from above and below, i.e.,
\begin{equation*}
  \frac12 u_2 |t| \leq g(t) \leq u_m |t| \qquad\text{for all }t\in \R.
\end{equation*}
Since $g$ is finite (and hence proper), convex, and continuous, the
corresponding integral operator $G:L^r(\omega_c)\to \R$ is finite,
convex, and continuous (and hence \emph{a fortiori} weakly lower
semi-continuous) for any $r\in [1,\infty]$; see, e.g.,
\cite[Prop.~2.53]{Barbu}.
Also, the properties of $g$ imply that
\begin{enumerate}[label={(\textsc{g}\arabic*)}]
  \item $G(v) > G(0)=0$ for all $v\in L^1(\omega_c)\setminus\{0\}$,
    \label{ass:g:bound}
  \item $\frac12 u_2 \norm{v}_{L^1(\Omega)}\leq G(v) \leq u_m
    \norm{v}_{L^1(\omega_c)}$ for all $v\in L^1(\omega_c)$.\label{ass:g:equiv}
\end{enumerate}

\section{The state equation}\label{sec:state}

We first consider the state equation for a fixed coefficient $u\in \hat U$ (i.e., defined and uniformly bounded on the full domain $\Omega$ and satisfying $\underline{u}\leq u \leq \overline{u}$ almost everywhere).
For given $u\in \hat U$, $f\in L^2(I,H)$, $y_0\in V$, and $y_1\in H$, we call $y=y(u)$ a (weak) solution to \eqref{eq:wave-equation} if
$y\in W:= L^2(I,V) \cap W^{1,2}(I,H)$ and
\begin{equation}\label{eq:weak}
  \left\{
    \begin{aligned}
      \int_0^T - (\partial_t y, \partial_t v)_H + (u \nabla y(t), \nabla v(t))_{\LL^2(\Omega)} dt &= \int_0^T (f(t), v(t))_H dt + (y_1,v(0))_H,\\
      y(0) &= y_0,
    \end{aligned}
  \right.
\end{equation}
for all $v\in W$ with $v(T) = 0$.
We then have the following existence and natural regularity result.
\begin{lemma}\label{lem:state_existence}
  For every $u \in \hat U$ and $(f, y_0, y_1) \in L^2(I,H) \times V \times H$, there exists a unique (weak) solution $y = y(u) \in Z := C(\overline V) \cap C^1(\overline H)$ to \eqref{eq:wave-equation} satisfying
  \begin{equation}\label{eq:apriori}
    \norm{y}_{C(\overline V)} + \norm{\partial_{t} y}_{C(\overline H)} + \norm{\partial_{tt} y}_{L^2({I}, V^*)} \le C_1(\norm{f}_{L^2(I,H)} + \norm{y_0}_V + \norm{y_1}_H)
  \end{equation}
  for a constant $C_1$ independent of $(f,y_0, y_1) \in L^2(I,H) \times V \times H$ and $u \in \hat U$.
\end{lemma}
\begin{proof}
  Except for the estimate on $\norm{\partial_{tt} y}_{L^2({I}, V^*)}$, the claim follows from \cite[Theorem 3.8.2, page 275]{LiMa72}, where we observe that due to our assumption on $u\in \hat U$, the energy is coercive with respect to the seminorm in $V$; see also \cite[Theorem 2.4.5]{Stolk00}. The constant $C_1$ depends on $\underline{u}$ and $\overline u$, but is otherwise independent of $u\in \hat U$.

  To verify the missing estimate, we use from \eqref{eq:apriori} that
  \begin{equation*}
    \norm{y}_{L^2({I}, V)}\le C_1(\norm{f}_{L^2(I,H)} + \norm{y_0}_V + \norm{y_1}_H).
  \end{equation*}
  Since $u \in \hat U$, we deduce that
  \begin{equation*}
    \norm{\div(u\nabla y)}_{L^2({I}, V^*)}\le C_1(\norm{f}_{L^2(I,H)} + \norm{y_0}_V + \norm{y_1}_H).
  \end{equation*}
  We further deduce from the state equation that
  \begin{equation*}
    \norm{\partial_{tt} y}_{L^2(I,V^*)} \le \tilde C_1(\norm{f}_{L^2(I,H)} + \norm{y_0}_V + \norm{y_1}_H),
  \end{equation*}
  with $\tilde C_1 = C_1+1$.
\end{proof}

By the change of variables $t\mapsto T-t$, we can also apply \cref{lem:state_existence} to the dual problem
\begin{equation}\label{eq:dual}
  \left\{
    \begin{aligned}
      \int_0^T - (\partial_t \phi, \partial_t v)_H + (u \nabla \phi(t), \nabla v(t))_{\LL^2(\Omega)} dt &= \int_0^T (g(t), v(t))_H dt + (\phi_1(T),v(T))_H,\\
      \phi(T) &= \phi_0 \text{ in } V,
    \end{aligned}
  \right.
\end{equation}
for any $g\in L^2(I,H)$, $\phi_0\in V$, $\phi_1 \in H$, and any $v\in W$ with $v(0)=0$.
\begin{corollary}\label{cor:dual_existence}
  For every $u\in \hat U$ and $g\in L^2(I,H)$, $\phi_0\in V$, and $\phi_1 \in H$, there exists a unique solution $\phi \in Z$ to \eqref{eq:dual} satisfying
  \begin{equation*}
    \norm{\phi}_{C(\overline I,V)} + \norm{\partial_t \phi}_{C(\overline I,H)}+\norm{\partial_{tt} \phi}_{L^2({I}, V^*)} \leq C_1 (\norm{g}_{L^2(I,H)} + \norm{\phi_0}_V + \norm{\phi_1}_{H}).
  \end{equation*}
\end{corollary}

Using this result, we can apply an Aubin--Nitsche trick or duality argument to show Lipschitz continuity of $u\mapsto y(u) \in L^2(I,H)$, which we will need to show differentiability of the tracking term later.
\begin{lemma}\label{lem:state_lipschitz}
  There exists a constant $L>0$ such that the mapping $u\mapsto y(u)$ satisfies
  \begin{equation*}
    \norm{y(u_1)-y(u_2)}_{L^2(I,H)}\leq L \norm{u_1-u_2}_{L^\infty(\Omega)} \quad \text{for all }u_1,u_2\in \hat U.
  \end{equation*}
\end{lemma}
\begin{proof}
  Let $u_1,u_2\in \hat U$ be arbitrary and set $\delta u := u_1-u_2$ and $\delta y:= y(u_1)-y(u_2)$. Subtracting the weak equations for $y(u_1)$ and $y(u_2)$, we have that $\delta y \in Z$ satisfies $\delta y(0) = 0$ and
  \begin{equation}\label{eq:difference}
    \int_0^T - (\partial_t \delta y, \partial_t v)_H + (u_1\nabla \delta y,\nabla \delta v)_{\LL^2(\Omega)} dt = -\int_0^T (\delta u\nabla y(u_2), \nabla v)_{\LL^2(\Omega)} dt
  \end{equation}
  for all $v\in W$ with $v(T) = 0$.

  Let now $g\in L^2(I,H)$ be arbitrary and consider the corresponding solution $\phi_g\in W$ of \eqref{eq:dual} for $u=u_1$, $g_0=0$, and $g_1=0$. Noting that $v=\phi_g$ is a valid test function for \eqref{eq:difference} and $w=\delta y$ is a valid test function for \eqref{eq:dual}, we obtain that
  \begin{equation*}
    (\delta y, g)_{L^2(I,H)}= -(\partial_t \delta y,\partial_t \phi_g)_{L^2(I,H)} + (u_1\nabla \delta y,\nabla \delta \phi_g)_{L^2(I,\LL^2(\Omega))} = (\delta u \nabla y(u_2),\nabla \phi_g)_{L^2(I,\LL^2(\Omega))}.
  \end{equation*}
  Using that $\delta u\in L^\infty(\Omega)$ together with \cref{lem:state_existence,cor:dual_existence}, this implies that
  \begin{equation*}
    \begin{aligned}
      (\delta y, g)_{L^2(I,H)} &\leq \norm{\delta u}_{L^\infty(\Omega)} \norm{y(u_2)}_{L^2(I,V)}\norm{\phi_g}_{L^2(I,V)}\\
      &\leq \norm{\delta u}_{L^\infty(\Omega)}
      C_1(\norm{f}_{L^2(I,H)}+\norm{y_0}_V+\norm{y_1}_H)\norm{g}_{L^2(I,H)}
    \end{aligned}
  \end{equation*}
  for all $g\in L^2(I,H)$. Since $L^2(I,H)$ is a Hilbert spaces, taking the supremum over all $g\in L^2(I,H)$ yields the claim.
\end{proof}

In stronger norms, we only have the following weak continuity result, which will be used repeatedly.
\begin{lemma}\label{lem:state_continuity}
  Let $\{u_n\}_{n\in\N}\subset \hat U$ be a sequence with $u_n\to u$ in $L^r(\Omega)$ for some $r\in [1,\infty)$. Then $u\in \hat U$ and $y(u_n)\wkto y(u)$ in $L^2(I,V)\cap W^{1,2}(I,H)\cap W^{2,2}(I,V^*)$. Furthermore, $y(u_n)\wkto y(u)$ in $H$ pointwise for all $t\in [0,T]$.
\end{lemma}
\begin{proof}
  The first assertion follows from the fact that $\hat U$ is closed in $L^r(\Omega)$. From $u_n\in \hat U$ and \cref{lem:state_existence}, the corresponding sequence $\{y(u_n)\}_{n\in \N}$ of solutions to \eqref{eq:weak} is well-defined and bounded in $L^2(I,V)\cap W^{1,2}(I,H)\cap W^{2,2}(I,V^*)$. By passing to successive subsequences (which we do not distinguish), we thus obtain that
  \begin{equation*}
    y_n \rightharpoonup \bar{y} \quad\text{ weakly in }\quad L^2(I,V) \cap W^{1,2}(I,H)\cap W^{2,2}(I,V^*).
  \end{equation*}
  From \eqref{eq:weak}, we in particular have that
  \begin{equation*}
    \int_{0}^{T} \, -(\partial_{t} y_n(t),\partial_{t} \phi(t))_H + (u_n \nabla y_n(t), \nabla \phi(t))_{\LL^2(\Omega)}\dt=\int_0^T(f(t), \phi(t))_H \dt + (y_1,\phi(0))_H
  \end{equation*}
  for arbitrary $\phi \in W \cap L^2(I,H^3(\Omega))$ with $\phi(T) =0$.
  Since $u_n\to u$ strongly in $L^r(\Omega)$ and $u_n,u\in \hat U\subset L^\infty(\Omega)$, we have for $r\in [1,2]$
  \begin{equation*}
    \|u_n-u\|_{H}\leq \|u_n-u\|_{L^\infty(\Omega)}^{(2-r)/2}\|u_n-u\|_{L^r(\Omega)}^{r/2}\leq C\|u_n-u\|_{L^r(\Omega)}^{r/2}
  \end{equation*}
  and thus for all $r\in [1,\infty)$
  \begin{multline*}
    \left|\int_0^T(u_n\nabla y_n-u\nabla \bar{y}, \nabla \phi)_{\LL^2(\Omega)}\dt\right|\\
    \begin{aligned}[t]
      &\leq \left|\int_0^T((u_n-u)\nabla \phi,\nabla y_n)_{\LL^2(\Omega)}\dt\right|+\left|\int_0^T(\nabla(y_n-\bar y),u\nabla \phi)_{\LL^2(\Omega)}\dt\right|\\
      &\leq C\|y_n\|_{L^2(I,V)}\|u_n-u\|_H\|\nabla \phi\|_{L^2(I,C(\overline \Omega))}+\left|\int_0^T(\nabla(y_n-\bar y),u\nabla \phi)_{\LL^2(\Omega)}\dt\right|\to 0.
    \end{aligned}
  \end{multline*}
  Then we can pass to the limit in the weak formulation to obtain
  \begin{equation*}
    \int_{0}^{T}- \, (\partial_{t} \bar{y}(t), \partial_{t} \phi(t))_H + (u \nabla \bar{y}(t), \nabla \phi(t))_{H} - (f(t), \phi(t))_H \dt - (y_1,\phi(0))_H = 0
  \end{equation*}
  for all such $\phi$. Since $u\in \hat U$ and since the set of functions with $\phi \in W \cap L^2(I,H^3(\Omega))$ and $\phi(T) =0$ is dense in $\{\varphi \in W: \varphi(0)=0\}$, the last equation also holds for all $v\in W$ with $v(T) =0$. The density can be shown by adapting the density argument of $C^\infty(\overline \Omega)$ in $V$; see, e.g., \cite[Cor.~9.8]{Brezis:2010a}.

  It remains to show that the limit $y$ satisfies the initial condition $y(0)=y_0$.
  First, for each $v\in H$ we have
  $(y_n(\cdot),v)_H \rightharpoonup (\bar y(\cdot),v)_H$ in $W^{1,2}(I)$. Hence
  \begin{equation*}
    (y_n(\cdot),v)_H \to (\bar y(\cdot),v)_H \quad\text{ in } C(\overline I).
  \end{equation*}
  due to the compact embedding of $W^{1,2}(I)$ to $C(\overline I)$.
  In particular this implies that $(y_n(0),v)_H=(y_0,v)_H =(\bar y(0),v)_H$ for all $v\in H$. Since $v\in H$ was arbitrary, this implies that $\bar y(0) =y_0$.
  This implies that $y=y(u)$, and since the solution of \eqref{eq:weak} is unique, a subsequence--subsequence argument shows that the full sequence converges weakly to $y(u)$.  By a similar argument, $y_n(t) \wkto \bar y(t)$ in $H$ for all $t\in \overline I$.
\end{proof}

Stronger continuity of $u\mapsto y(u)$ can be shown with respect to the $L^\infty$ topology for the controls.
\begin{lemma}\label{lem:state_strong}
  Assume that $f\in L^2(I,V)$ and let $\{u_n\}_{n\in\N}\subset \hat U$ be a sequence with $u_n\to u$ in $L^\infty(\Omega)$. Then $y(u_n)\to y(u)$ in $L^2(I,V)\cap W^{1,2}(I,H)$.
\end{lemma}

\begin{proof}
  First, the embedding $L^\infty(\Omega)\subset L^r(\Omega)$, $r\in [1,\infty)$, for bounded $\Omega$ together with \cref{lem:state_continuity} shows that $y_n:= y(u_n)\wkto y(u)$ in $L^2(I,V)\cap W^{1,2}(I,H)\cap W^{2,2}(I,V^*)$.

  We now introduce for $u\in \hat U$ and $y:=y(u)$ the energy
  \begin{equation*}
    \mathcal{E}_u(t) := \frac12 \scalprod{u\nabla y(t),\nabla y(t)}_{\LL^2(\Omega)} + \frac12\norm{\partial_t y(t)}_{H}^2 + \frac12\norm{y(t)}_H^2 \quad\text{for a.e.}~ t\in I.
  \end{equation*}
  By the Lions--Magenes Lemma (\cite[Lem.~8.3]{LiMa72}, cf.~also \cite[(2.24), p.~24]{Stolk00}), we have that
  \begin{equation*}
    \frac{d}{dt} \mathcal{E}_u(t) = \scalprod{f(t),\partial_t y(t)}_{H} + \scalprod{y(t),\partial_t y(t)}_H\quad\text{in }L^1(I),
  \end{equation*}
  and thus by the fundamental theorem of calculus we find that
  \begin{equation}\label{eq:lionsmagenes}
    \mathcal{E}_u(t) = \mathcal{E}_u(0) + \int_0^t \scalprod{f(s),\partial_t y(s)}_{H} + \scalprod{y(s),\partial_t y(s)}_H\ds.
  \end{equation}

  We now define for $v\in V$
  \begin{equation*}
    \norm{v}_{V_u}^2:=\scalprod{u\nabla v ,\nabla v}_{\LL^2(\Omega)} + \norm{v}_H^2,
  \end{equation*}
  which is an equivalent norm on $V$ for any $u\in \hat U$.
  Subtracting \eqref{eq:lionsmagenes} for $\mathcal{E}_u$ and $\mathcal{E}_{u_n}$ and adding the productive zero then yields for almost every $t\in I$ that
  \begin{multline}\label{eq:state_strong1}
    \frac12 \norm{y(t)}_{V_u}^2 + \frac12\norm{\partial_t y(t)}_H^2
    -\left(\frac12 \norm{y_n(t)}_{V_u}^2 + \frac12\norm{\partial_t y_n(t)}_H^2\right) \\
    =
    \frac12\scalprod{(u_n-u)\nabla y_n(t),\nabla y_n(t)}_{\LL^2(\Omega)} +
    \frac12\scalprod{(u-u_n)\nabla y_0,\nabla y_0}_{\LL^2(\Omega)} \\
    + \int_0^t \scalprod{f(s),\partial_t y(s)-\partial_t y_n(s)}_H + \scalprod{y(s)-y_n(s),\partial_t y(s)}_H + \scalprod{y_n(s),\partial_t y(s)-\partial_t y_n(s)}_H \ds,
  \end{multline}
  and hence that
  \begin{multline*}
    \left|\frac12 \norm{y(t)}_{V_u}^2 + \frac12\norm{\partial_t y(t)}_H^2
    -\left(\frac12 \norm{y_n(t)}_{V_u}^2 + \frac12\norm{\partial_t y_n(t)}_H^2\right)\right|
    \leq \norm{u_n-u}_{L^\infty(\Omega)} \norm{y_n}_{C(\overline I,V)}^2\\
    +(\|f\|_{L^2(I,V)}+\|y_n\|_{L^2(I,V)})\|\partial_t y-\partial_ty_n\|_{L^2(I,V^\ast)}+\|\partial_t y\|_{L^2(I,H)}\|y-y_n\|_{L^2(I,H)}.
  \end{multline*}
  We know from \cref{lem:state_continuity} that $y_n\rightharpoonup y$ in $L^2(I,V)\cap W^{1,2}(I,H)\cap W^{2,2}(I,V^\ast)$. Thus the Aubin--Lions Lemma and the compactness of the embeddings $V\hookrightarrow H$ and $H\hookrightarrow V^\ast$ imply that $y_n \to y$ in $L^2(I,H)\cap W^{1,2}(I,V^\ast)$.
  Thus we have
  \begin{equation}\label{eq:state_strong2}
    \int_0^T\norm{y_n(t)}_{V_u}^2+\norm{\partial_t y_n(t)}_{H}^2\dt \to \int_0^T\norm{y(t)}_{V_u}^2+\norm{\partial_t y(t)}_{H}^2\dt.
  \end{equation}
  Since $\norm{\cdot}_{V_u}$ is an equivalent norm on $V$, we have that the normed vector space $V_u := (V,\norm{\cdot}_{V_u})$ is equivalent to $V$ and hence that $L^2(I,V_u)\cap W^{1,2}(I,H)\simeq L^2(I,V)\cap W^{1,2}(I,H)$. This implies that $y_n\wkto y$ also in $L^2(I,V_u)\cap W^{1,2}(I,H)$, and together with \eqref{eq:state_strong2} the Radon--Riesz property of Hilbert spaces implies that $y_n\to y$ strongly in $L^2(I,V_u)\cap W^{1,2}(I,H)$. Appealing again to the equivalence of $V$ and $V_u$ then yields the claim.
\end{proof}

\bigskip

Under additional assumptions, we can show an improved regularity result.
\begin{assumption}\label{ass:data} The data satisfy $f\in L^2(I;V)$ and $(y_0,y_1)\in H^2(\Omega)\times H^1(\Omega)$ with $\partial_\nu y_0 =0$. Furthermore,
  \begin{enumerate}
    \item $\hat u$ is constant on $\Omega\setminus \omega_c$ and
    \item $y_0$ is constant on $\omega_c$ and $\overline{\omega_c} \subset \Omega$.
  \end{enumerate}
\end{assumption}

The following result will be used to show Fréchet differentiability of the tracking term in \cref{lem:tracking_frechet} below.
\begin{proposition} \label{prop:state_regularity}
  Let $u\in \hat U$ and \cref{ass:data} hold.
  Then there exists $q > 2$ and a constant $\hat C$ independent of $u$ such that $y(u) \in {L^\infty(I,W^{1,q}(\Omega))}$ and
  for all $u \in \hat U$,
  \begin{equation}\label{eq:state-regularity}
    \norm{y(u)}_{L^\infty(I,W^{1,q}(\Omega))} \le \hat C\left(\norm{y_1}_{V} + \norm{y_0}_{H^2(\Omega)}+ \norm{f}^2_{L^2(I,V)}\right).
  \end{equation}
\end{proposition}
\begin{proof} We proceed in two steps.

  \emph{Step 1.} First we assume that additionally
  \begin{equation}\label{eq:ass-extra-reg}
    f \in W^{2,2}(I,V),\qquad y_0 \in H^3(\Omega) \quad\text{with}\quad \partial_\nu y_0 = 0,
    \qquad\text{and }\quad y_1 \in H^2(\Omega).
  \end{equation}
  Let $u\in \hat U$ and approximate $u$ by $\{u_n\} \subset C^\infty(\overline \Omega) \cap \hat U$ with $u_n\to u$ in $L^r(\Omega)$ for some $r\in [1,\infty)$ with $u_n=\hat u$ in $\Omega\setminus \omega_c$.
  Such a sequence can be found by first approximating $u$ by $\tilde u_n \in C^\infty(\mathbb{R}^d)$ and $\tilde u_n = \hat u$ in $\Omega\setminus \omega_c$; this sequence in turn is constructed by first introducing an intermediate approximation of functions $\hat u_n$ with the property that $\lim_{n\to \infty} \hat u_n =u$ in $L^r(\Omega)$ and $\hat u_n= \hat u$ in $\Omega \setminus \hat\omega_{c,n}$, where the closure of $\hat \omega_{c,n}$ is contained in $\omega_c$ and $0<\mathrm{dist}(\partial \omega_c, \hat \omega_{c,n}) \le n^{-1}$.
  Then we use convolution by mollifiers of the functions $\hat u_n$ to obtain functions $\tilde u_n\in C^\infty(\mathbb{R}^d)$ that satisfy $\lim_{n\to \infty} \tilde u_n =u$ in $L^r(\Omega)$, see e.g. \cite[page 132]{Grigoryan:2009}, and $\tilde u_n=\hat u$ in $\Omega\setminus \omega_c$.
  Next we choose functions $\underline{\varphi}^n$ and $\overline \varphi _n$ in $C^\infty(\mathbb{R})$ with a Lipschitz constant $L$ and such that $\underline{u}\le\underline{\varphi}^n $, $\overline \varphi_n\le \overline u$, and $\underline{\varphi}^n(s)\to \max(\underline u,s)$, $\overline \varphi_n(s) \to \min (\overline u,s)$ for all $s\in\mathbb{R}$; see \cite[page 125]{Grigoryan:2009}.
  We then set $u_n= \overline \varphi_n(\underline{\varphi}^n (\tilde u_n))$, and estimate
  \begin{equation*}
    \begin{aligned}
      \|u-u_n\|_{L^r(\Omega)} &\le \|u-\overline \varphi_n(\underline{\varphi}^n (u))\|_{L^r(\Omega)} + \|\overline \varphi_n(\underline{\varphi}^n (u)) - \overline \varphi_n(\underline{\varphi}^n (u_n)) \|_{L^r(\Omega)}\\
      &\le \|u-\overline \varphi_n(\underline{\varphi}^n(u))\|_{L^r(\Omega)} + L^2 \|u-\tilde u_n\|_{L^r(\Omega)} \to 0,
    \end{aligned}
  \end{equation*}
  using Lebesgue's bounded convergence theorem.

  We replace $u$ in \eqref{eq:wave-equation} by $u_n$. Due to the regularity assumptions \eqref{eq:ass-extra-reg} and the assumption that $y_0$ is constant on $\omega_c$, we have
  \begin{equation*}
    y_0 \in H^3(\Omega), \quad y_1 \in H^2(\Omega), \quad f(0) + \div (u_n \nabla y_0) \in V, \quad f'(0) + \div (u_n \nabla y_1) \in H.
  \end{equation*}
  Together with ${\partial_\nu y_0}= 0$ and $f \in W^{2,2}(I,V) \subset W^{2,2}(I,H)$, these properties allow applying Theorem 30.3 (with $k=3$) and Theorem 30.4 in \cite{Wloka},
  which guarantee that $y_n = y(u_n) \in W^{1,2}(I,H^2(\Omega))\cap W^{2,2}(I,V)$ and ${\partial_\nu y_n}(t) = 0$ on $\partial \Omega$ for $t \in \overline{I}$.
  Then we multiply \eqref{eq:wave-equation} with $-{\partial_t} \div (u_n\nabla y_n(t))$ and integrate over $\Omega$. Integrating by parts on the right-hand side and using that $\partial_\nu\partial_ty_n=0$ on $\partial\Omega$,
  we obtain
  \begin{equation*}
    (u_n \nabla \partial_{tt} y_n(t), \nabla \partial_t y_n(t))_{\LL^2(\Omega)} +\frac{1}{2} \partial_t \norm{\div (u_n\nabla y_n(t))}^2_H = (\nabla f(t), u_n \nabla \partial_t y_n(t))_{\LL^2(\Omega)}
  \end{equation*}
  and thus
  \begin{equation*}
    {\partial_t} \norm{ \sqrt{u_n} \nabla \partial_t y_n(t)}^2_{\LL^2(\Omega)} + \partial_t \norm{\div (u_n \nabla y_n(t))}^2_H \le \norm{\sqrt{u_n}\nabla f(t)}^2_{\LL^2(\Omega)} + \norm{\sqrt{u_n} \nabla \partial_t y_n(t)}^2_{\LL^2(\Omega)}.
  \end{equation*}
  Integrating this expression on $(0,t)$, we find for $t \in (0,T]$ that
  \begin{multline*}
    \norm{\sqrt{u_n} \nabla \partial_t y_n(t)}^2_{\LL^2(\Omega)} + \norm{\div (u_n \nabla y_n(t))}^2_H \\
    \le \norm{\sqrt{u_n} \nabla y_1}^2_{\LL^2(\Omega)} + \norm{\div (u_n \nabla y_0)}^2_H + \int_{0}^{t} \norm{\sqrt{u_n}\nabla f(s)}^2_{\LL^2(\Omega)} \ds + \int_{0}^{t} \norm{\sqrt{u_n} \nabla \partial_t y_n(s)}^2_{\LL^2(\Omega)} \ds.
  \end{multline*}
  Gronwall's inequality then implies that for each $t \in (0,T]$,
  \begin{equation*}
    \norm{\div (u_n \nabla y_n(t))}^2_H \le \left(\norm{\sqrt{u_n} \nabla y_1}^2_{\LL^2(\Omega)}+ \norm{\div (u_n \nabla y_0)}^2_H + \overline{u}\norm{f}^2_{L^2(I,V)}\right) e^{T}.
  \end{equation*}
  Since $y_1 \in H^2(\Omega)$, it follows that $\{\sqrt{u_n} \nabla y_1\}_{n\in\N}$ is bounded in $\LL^2(\Omega)$.
  Moreover, $\div (u_n \nabla y_0)|_{\Omega \setminus \omega_c} = \div (\hat u \nabla y_0)|_{\Omega \setminus \omega_c} = \hat u \Delta y_0|_{\Omega \setminus \omega_c}$ and $\div (u_n \nabla y_0)|_{\omega_c} =0$. Hence $\{\div (u_n \nabla y_0)\}_{n\in\N}$ is bounded in
  $H$. We can thus conclude that $\{\div (u_n \nabla y_n)\}_{n\in\N}$ is bounded in $L^\infty(I, H)$.

  Our next aim is to obtain $L^\infty(I,W^{1,q}(\Omega))$ regularity and boundedness for $y_n$ for some $q > 2$. For this purpose, we define for some $\lambda > 0$
  \begin{equation} \label{eq:regularity-gn}
    g_n := -\div(u_n \nabla y_n) + \lambda y_n
  \end{equation}
  and note that $\{g_n\}_{n\in\N}$ is bounded in $L^\infty(I,H)$.
  Furthermore, Sobolev's embedding theorem implies that $W^{1,s'}(\Omega) \hookrightarrow L^2(\Omega)$ for every $s' \ge 1$ in case $d = 2$, and for every $s' \ge \frac{6}{5}$ in case $d = 3$. Following the notation of \cite{Groeger}, we denote by $W^{-1,s}(\Omega)$ the dual space of $W^{1,s'}(\Omega)$ with $s$ the conjugate of $s'$. Then we have $H \hookrightarrow W^{-1,s}(\Omega)$, where $s \in [1, \infty)$ for $d = 2$ and $s \in [1,6]$ for $d = 3$. It follows that $\{g_n\}_{n\in\N}$ is bounded in $L^\infty(I,W^{-1,s}(\Omega))$.
  Considering now \eqref{eq:regularity-gn} (together with homogeneous Neumann boundary conditions) for a.e. $t\in I$ as an equation for $y_n(t)$, this implies that there exists some $q > 2$ such that
  \begin{equation}\label{eq:regularity-yn}
    \|y_n(t)\|_{W^{1,q}(\Omega)}\leq C\|g_n\|_H\leq \hat C\left(\norm{y_1}_{V} + \norm{y_0}_{H^2(\Omega)} + \norm{f}^2_{L^2(I,V)}\right),
  \end{equation}
  where the constant $C$ depends only on $\underline u$, $\overline u$ and $q$, but not on $t$; see \cite[Thm.~1]{Groeger}.
  Hence $\{y_n\}_{n\in\N}$ is bounded in $L^\infty(I,W^{1,q}(\Omega))$.
  Since $L^1(I)$ and $W^{1,q}(\Omega)$ are separable with the latter being reflexive, $L^\infty(I,W^{1,q}(\Omega))$ is the dual of a separable space; see, e.g., \cite[Thm.~8.18.3]{Ed65}. Hence there exists a subsequence with $y_n \rightharpoonup^* \bar{y} \in L^\infty(I,W^{1,q}(\Omega))$.

  Finally, from \cref{lem:state_continuity} we also have that $y_n\wkto y(u)$ in $L^2(I,V)\cap W^{1,2}(I,H)$ and hence, by uniqueness of $y(u)$, that $y(u)=\bar y \in L^\infty(I,W^{1,q}(\Omega))$. Using weak$^*$ semi-continuity of norms (cf., e.g., \cite[p.~63]{Brezis:2010a}), we can now pass to the limit in \eqref{eq:regularity-yn} to obtain \eqref{eq:state-regularity}, for those $(y_0,y_1,f)$ which satisfy the additional regularity assumption \eqref{eq:ass-extra-reg}.

  \bigskip

  \emph{Step 2.} We relax the requirements on the problem data and choose an arbitrary $(y_0,y_1,f)\in X:= H^2(\Omega)\times H^1(\Omega)\times L^2(I;V)$, with $\partial_\nu y_0 =0$. Then there exists $(y^n_0,y^n_1,f^n)\in H^3(\Omega)\times H^2(\Omega)\times W^{2,2}(I;V)$, with $\partial_\nu y^n_0 =0$ such that
  $\lim_{n\to \infty}(y^n_0,y^n_1,f^n)= (y_0,y_1,f) $ in $X$. As this is standard for the second and third component, we only address the first one.
  Let $\omega_c^c := \Omega \setminus \overline {\omega_c}$. By assumption, $\partial \Omega \cap \partial \omega_c= \emptyset$; in addition, $\Omega$ and $\omega_c$ are $C^{2,1}$ domains, and thus $\omega_c^c$ is a $C^{2,1}$ domain as well. It is in this step that the $C^{2,1}$ regularity of the domains is used.
  Since $y_0 \in H^2(\Omega)$, we have $y_0|_{\partial \Omega} \in H^{3/2}(\partial \Omega)$. Moreover, $y_0|_{\partial \omega_c}=:y_c$ is constant and $\partial_\nu y_0|_{\partial \omega^c_c} = \partial_\nu y_0|_{\partial \omega_c \cup \partial \Omega} =0$.
  Let $\tilde v_n \in H^{5/2}(\partial \Omega)$ be such that $\tilde v_n \to y_0|_{\partial \Omega}$ in $H^{3/2}(\partial \Omega)$. Accordingly let $v_n \in H^3(\omega^c_c)$ with $v_n|_{\partial \omega_c}=y_c$, $\partial_\nu v_n |_{\partial \omega^c_c}=0$, $\partial_{x_i x_j}v_n|_{\partial\omega_c}=0$, $v_n|_{\partial \Omega} = \tilde v_n$, and $v_n\to v$ in $H^2(\omega^c_c)$, where $v\in H^2(\omega^c_c)$ satisfies $v|_{\partial \omega_c}=y_c$, $\partial_\nu v|_{\partial \omega^c_c}=0$,
  and $v|_{\partial \Omega} = y_0|_{\partial \Omega}$.
  Denoting by $v_{n,\mathrm{ext}}$ and $v_{\mathrm{ext}}$ the extensions of $v_n$ and $v$ by the constant $y_c$ on $\omega_c$, we have $v_{n,\mathrm{ext}} \to v_{\mathrm{ext}} \in H^2(\Omega)$, $\partial_\nu v_{n,\mathrm{ext}}|_{\partial \Omega}=0$, $v_{n,\mathrm{ext}}|_{\omega_c}= v_{\mathrm{ext}}|_{\omega_c}=y_c$, and $v_{n,\mathrm{ext}} \in H^3(\Omega) $, where we use that $\partial_{x_i x_j}v_n|_{\partial\omega_c}=0$.
  Next, observe that $y_0-v \in H^2_0(\omega^c_c)$. This implies the existence of functions $w_n\in C^\infty(\omega_c^c)$ with compact support in $\omega^c_c$ such that $w_n \to y_0-v$ in $H^2(\omega_c^c)$; see, e.g., \cite[pp. 17 and 31]{Grisvard}. Denoting the extension by zero to $\omega_c$ of $w_n$ by $w_{n,\mathrm{ext}}$, we have $w_{n,\mathrm{ext}}\in H^3(\Omega), w_{n,\mathrm{ext}} \to y_0-v_{\mathrm{ext}}$ in $H^2(\Omega)$, and $\partial_\nu w_{n,\mathrm{ext}}=0|_{\partial \Omega}$. Finally, the sequence $y^n_0=v_{n,\mathrm{ext}}+w_{n,\mathrm{ext}}$ defines the desired approximation of $y_0$ such that $y^n_0|_{\omega_c}=y_c$, $\partial_\nu y_n|_{\partial \Omega}=0$, and $y^n_0 \to y_0$ in $H^2(\Omega)$.

  We next define $y^n$ as the solution to \eqref{eq:wave-equation} with data $(y_0^n,y_1^n,f^n)$. From \cref{lem:state_existence} and \eqref{eq:state-regularity} for the smooth data $(y_0^n,y_1^n,f^n)$ we deduce that $y^n \rightharpoonup y$ in $L^2(I,V)\cap W^{1,2}(I,H)$ and $y^n \wkto^* y$ in $L^\infty(I,W^{1,q}(\Omega))$, where $y$ is the solution of \eqref{eq:wave-equation} with data $(y_0,y_1,f)$, and
  \begin{equation*}
    \norm{y^n}_{L^\infty(I,W^{1,q}(\Omega))} \le \hat C\left(\norm{y^n_1}_{V} + \norm{y^n_0}_{H^2(\Omega)}+ \norm{f^n}^2_{L^2(I,V)}\right)
  \end{equation*}
  for all $n$. Passing to the limit as $n\to \infty$, we obtain \eqref{eq:state-regularity} for $(y_0,y_1,f)\in X$ with $\partial_\nu y_0=0$.
\end{proof}

\begin{remark}
  If \cref{ass:data} holds, the requirement in \cref{lem:state_strong} on the convergence of $u_n$ can be relaxed to $u_n\to u$ in $L^{\frac{q}{q-2}}(\Omega)$, where $q$ is the exponent from \cref{prop:state_regularity}. In this case, the first term on the right-hand side in \eqref{eq:state_strong1} can be estimated by Hölder's inequality as
  \begin{equation*}
    \int_{\Omega} |u_n-u| |\nabla y_n(t)|^2 \dx  \leq \norm{u_n-u}_{L^{\frac{q}{q-2}}(\Omega)}\norm{y_n(t)}_{L^\infty(I,W^{1,q}(\Omega))}^2\to 0,
  \end{equation*}
  where we used \cref{prop:state_regularity}. Then again $y(u_n) \to y(u)$ in $W$.
\end{remark}

\begin{remark}
  In \cref{ass:data}\,(ii), the requirement $\overline {\omega_c} \subset \Omega$ was only used in Step 2 of the proof of \cref{prop:state_regularity}. It it is not necessary if instead $y_0\in H^3(\Omega)$ is assumed.
\end{remark}

\section{Existence and optimality conditions}\label{sec:existence}

Deriving useful optimality conditions requires replacing the pointwise control constraints with a differentiable approximation of a cutoff function. We thus introduce the superposition operator
\begin{equation*}
  \Phi_\eps: L^1(\omega_c) \to U, \qquad [\Phi_\eps(u)](x) := \phi_\eps(u(x)),
\end{equation*}
where $\phi_\eps:\R\to [u_{\min},u_{\max}]$ is such that $\Phi_\eps$ is Lipschitz continuous from $L^r(\omega_c)\to L^r(\omega_c)$ for every $r\in [1,\infty]$ and $\eps\geq 0$ and Fréchet differentiable from $L^\infty(\omega_c)\to L^\infty(\omega_c)$ for $\eps>0$ (and thus ensuring Fréchet differentiability of the tracking term; see \cref{lem:tracking_frechet} below).
The construction of such a $\phi_\eps$ and the characterization of the Fréchet derivative of $\Phi_\eps$ via pointwise a.e. multiplication can be carried out in the same way as in \cite[\S\,2.3]{CKK:2017}.

We then consider for $\eps\geq 0$ the reduced, unconstrained optimization problem
\begin{equation}\label{eq:problem}
  \min_{u\in BV(\omega_c)} J_\eps(u)
\end{equation}
for
\begin{equation*}
  J_\eps(u) := \frac 12 \|By(\hat E\Phi_\eps(u))-y_d\|_{\calO}^2+ \alpha G(u) + \beta TV(u)
\end{equation*}
for some $y_d \in \calO$ and $\alpha,\beta>0$, where $u \mapsto y(u)$ denotes the solution mapping of \eqref{eq:weak} introduced in the previous section and
$\hat E$ is the affine extension operator from $\omega_c$ to $\Omega$ defined in \eqref{eq:extension}.
We point out that the role of $\eps$ is not that of a smoothing parameter for the optimization problem, which remains nonsmooth for $\eps>0$ due to the presence of $G$ and $\TV$; it merely influences the behavior of the cutoff function near the upper and lower values of the pointwise bounds for the coefficient.

Existence of optimal controls now follows analogously to \cite[Prop.~3.1]{CKK:2017}.
\begin{proposition}\label{thm:existence:bv}
  For every $\eps\geq 0$, there exists a global minimizer $\bar u\in BV(\omega_c)\cap U$ of $J_\eps$.
\end{proposition}
\begin{proof}
  Since $J_\eps$ is bounded from below, there exists a minimizing sequence $\{u_n\}_{n\in\N}\subset BV(\omega_c)$. Furthermore, we may assume without loss of generality that there exists a $C>0$ such that
  \begin{equation*}
    C\left(\norm{u_n}_{L^1(\omega_c)} + \TV(u_n)\right) \leq J_\eps(u_n)\leq J_\eps(0) \quad\text{for all }n\in\N,
  \end{equation*}
  and hence that $\{u_n\}_{n\in\N}$ is bounded in $BV(\omega_c)$.
  By the compact embedding of $BV(\omega_c)$ into $L^1(\omega_c)$ for any $d\in\N$, we can thus extract a subsequence, denoted by the same symbol, converging strongly in $L^1(\omega_c)$ to some $\bar u\in L^1(\omega_c)$. Due to the continuity of $\Phi_\eps$ as well as of $\hat E$, we have $\hat E\Phi_\eps(u_n)\to \hat E\Phi_\eps(\bar u)\in U$ in $L^1(\omega_c)$.

  Lower semi-continuity of $G$ and $\TV$ with respect to the strong convergence in $L^1(\omega_c)$ and the weak convergence $y(\hat E\Phi_\eps(u_n))\wkto y(\hat E\Phi_\eps(\bar u))$ in $L^2(Q)$ from \cref{lem:state_continuity} yield that
  \begin{equation*}
    J_\eps(\bar u) \leq \liminf_{n\to\infty} J_\eps(u_n) \leq J_\eps(u) \qquad\text{for all }u\in BV(\omega_c)
  \end{equation*}
  and thus that $\bar u\in BV(\omega_c)$ is the desired minimizer.

  The fact that $\bar u\in U$ then follows by a contraposition argument based on Stampacchia's Lemma for $BV$ functions and the pointwise definition of $G$; see \cite[Prop.~3.2]{CKK:2017}.
\end{proof}
The convergence of minimizers of \eqref{eq:problem} as $\eps\to0$ can be shown along the same lines as indicated at the end of \cite[\S\,3]{CKK:2017}. 

\bigskip

We now derive first-order optimality conditions for the solution of \eqref{eq:problem}. To this end, we first show Fréchet differentiability of the tracking term
\begin{equation}\label{eq:tracking}
  F_\eps:L^\infty(\omega_c)\to \R,\qquad F_\eps(u) = \frac{1}2\norm{By(\hat E\Phi_\eps(u))-y_d}_{\calO}^2
\end{equation}
as in \cite[Lem.~4.1]{CKK:2017} by using for given $u\in L^\infty(\omega_c)$ and $y\in W$ the definition of the adjoint equation
\begin{equation}\label{eq:adjoint}
  \left\{
    \begin{aligned}
      \int_0^T - (\partial_t p, \partial_t \varphi)_H + (u \nabla p, \nabla \varphi)_{\LL^2(\Omega)} \dt &= \int_0^T (B^\ast(By-y_d), \varphi)_H \dt,\\
      p(T) &= 0 \text{ in } V,
    \end{aligned}
  \right.
\end{equation}
for any $\varphi\in W$ with $\varphi(0)=0$, which admits a unique solution $p\in W$ by \cref{lem:state_existence}. In the following, we use the regularity of solutions to identify the derivative in $L^\infty(\omega_c)^*$ with its representation in $L^1(\omega_c)$, considered as a subset of $L^\infty(\omega_c)^*$.
Since the extension operator $\hat E$ is affine, we also introduce the corresponding linear extension operator $\hat E_0 = \hat E':L^2(\omega_c)\to L^2(\Omega)$.
\begin{lemma}\label{lem:tracking_frechet}
  For every $\eps>0$, the mapping $F_\eps$ defined in \eqref{eq:tracking} is Fréchet differentiable in every $u\in L^\infty(\omega_c)$, and the Fréchet derivative is given by
  \begin{equation}\label{eq:deri}
    F_\eps'(u) = \hat E_0^*\left(\int_0^T\nabla y\cdot \nabla p\dt\right)\,\Phi_\eps'(u) \in L^1(\omega_c),
  \end{equation}
  where $y=y(\hat E\Phi_\eps(u))$ is the solution of \eqref{eq:weak}, $p$ is the solution of \eqref{eq:adjoint}, and $\hat E_0^*:L^2(\Omega)\to L^2(\omega_c)$ is the restriction operator.

  If \cref{ass:data} holds, $F_\eps'(u) \in L^{\frac{2q}{2+q}}(\omega_c)$ for the $q>2$ given in \cref{prop:state_regularity}.
\end{lemma}
\begin{proof}
  We first show directional differentiability. Let $w, h\in L^\infty(\omega_c)$ and $\rho>0$ be arbitrary. We define $\tilde y(w):=y(\hat \Phi_\eps(w))$. We now insert the productive zero $B\tilde y(w)-B\tilde y(w)$ in $F_\eps(w+\rho h)$ and expand the square to obtain
  \begin{equation}\label{eq:tracking_frechet1}
    \begin{aligned}[t]
      F_\eps(w+\rho h)-F_\eps(w) &= \frac{1}{2}\norm{B(\tilde y(w+\rho h)-\tilde y(w))+(B\tilde y(w)-y_d)}_{\calO}^2 - \frac{1}{2}\norm{B\tilde y(w)-y_d}_{\calO}^2\\
      &= \frac{1}{2}\norm{B(\tilde y(w+\rho h)-\tilde y(w))}_{\calO}^2 + (B(\tilde y(w+\rho h)-\tilde y(w)),B\tilde y(w)-y_d)_{\calO}.
    \end{aligned}
  \end{equation}
  For the first term, we can use \cref{lem:state_lipschitz}, the boundedness of $B$ and the Lipschitz continuity of $\Phi_\eps$ to estimate
  \begin{equation}\label{eq:tracking_frechet1a}
    \frac{1}{2}\norm{B(\tilde y(w+\rho h)-\tilde y(w))}_{\calO}^2 \leq C\rho^2\norm{h}_{L^\infty(\omega_c)}^2.
  \end{equation}
  For the second term in \eqref{eq:tracking_frechet1}, we introduce the adjoint state $p(w)$ and use the fact that $\delta y:=\tilde y(w+\rho h)-\tilde y(w)\in W$ with $\delta y(0)=0$. Testing \eqref{eq:adjoint} with $\phi= \delta y$, and using
  \eqref{eq:weak} for $y=\tilde y(w)$ and $y=\tilde y(w+\rho h)$, each time with $v=p(w)$, and inserting a productive zero, we find
  \begin{equation*}
    \begin{aligned}[t]
      \scalprod{\delta y,B^\ast(B\tilde y(w)-y_d)}_{L^2(Q)}&= -\scalprod{\partial_t p(w),\partial_t \delta y}_{L^2(Q)} + \scalprod{\hat E\Phi_\eps(w) \nabla p(w),\nabla \delta y}_{L^2(I,\LL^2(\Omega))}\\
      &= \scalprod{(\hat E\Phi_\eps(w+\rho h)-\hat E\Phi_\eps(w))\nabla \tilde y(w+\rho h),\nabla p(w)}_{L^2(I,\LL^2(\Omega))}.
    \end{aligned}
  \end{equation*}
  By \cref{lem:state_lipschitz} we have that $\tilde y(w+\rho h)\wkto \tilde y(w)$ in $L^2(I,V)$ as $\rho\to 0^+$. Moreover, since $\eps>0$ we have that $\hat E\Phi_\eps$ is Frechet differentiable in $L^\infty(\omega_c)$ with $\hat E_0\Phi_\eps'(w),\hat E_0\Phi_\eps'(w+\rho h)\in L^\infty(\Omega)$.  Hence, dividing \eqref{eq:tracking_frechet1} by $\rho>0$ and passing to the limit $\rho\to0$ implies in combination with \eqref{eq:tracking_frechet1a} that
  \begin{equation*}
    F_\eps'(w;h) := \lim_{\rho\to0^+}\frac1\rho(F_\eps(w+\rho h) - F_\eps(w)) = \dual{h,\Phi_\eps'(w)\hat E_0^* \int_0^T\nabla \tilde y(w)\cdot\nabla p(w)\dt}_{L^\infty(\omega_c),L^1(\omega_c)}.
  \end{equation*}
  Since the mapping $h\mapsto F_\eps'(w;h)$ is linear and bounded, $\hat E_0^*(\int_0^T\nabla y(w)\cdot \nabla p(w)\dt)\,\Phi_\eps'(w)$ is the Gâteaux derivative of $F_\eps$ at $w$. Thus, $F_\eps$ is Gâteaux differentiable in $L^\infty(\omega_c)$.

  It remains to show that this is also a Fréchet derivative. From the above, we have that
  \begin{multline*}
    \left|F_\eps(w+h)-F_\eps(w)-F_\eps'(w)h\right|\leq C\norm{h}_{L^\infty(\omega_c)}^2\\
    + \left|\scalprod{(\hat E\Phi_\eps(w+h)-\hat E\Phi_\eps(w))\nabla \tilde y(w+h)- \hat E_0\Phi_\eps'(w)h\nabla \tilde y(w),\nabla p(w)}_{L^2(I,\LL^2(\Omega))}\right|\\
    \leq C\norm{h}_{L^\infty(\omega_c)}^2+\left( \norm{\hat E\Phi_\eps(w+h)-\hat E\Phi_\eps(w)- \hat E_0\Phi_\eps'(w)h}_{L^\infty(\omega_c)}\norm{\nabla \tilde y(w+h)}_{L^2(I,\LL^2(\omega_c))}\right.\\
    \left.+ \|\hat E_0\Phi_\eps'(w)\|_{L^\infty(\omega_c)} \|\nabla \tilde y(w+h)-\nabla \tilde y(w)\|_{L^2(I,\LL^2(\Omega))}\right)\|\nabla p(w)\|_{L^2(I,\LL^2(\Omega))},
  \end{multline*}
  and hence that
  \begin{equation*}
    \frac{|F_\eps(w+h)-F_\eps(w)-F_\eps'(w)h|}{\norm{h}_{L^\infty(\omega_c)}} \to 0\quad\text{for}\quad\|h\|_{L^\infty(\omega_c)}\to0
  \end{equation*}
  since $\tilde y(w+h) \to \tilde y(w)$ in $L^2(I,V)$ by \cref{lem:state_strong}.

  The regularity follows from $\tilde y(w),p(w)\in L^2(I,V)$ together with the properties of the norm in Bochner spaces, cf.~\cite[Cor.~V.1]{Yosida:1980}.
  If \cref{ass:data} holds, \cref{prop:state_regularity} yields $\tilde y(w)\in L^\infty(I,W^{1,q}(\Omega))$ for some $q>2$ and hence $\hat E_0^*\left(\int_0^T\nabla \tilde y(w)\cdot \nabla p(w)\dt\right)\Phi_\eps'(w)\in L^2(I,L^{\frac{2q}{2+q}}(\omega_c))$.
\end{proof}

We can now proceed exactly as in \cite{CKK:2017} to obtain first-order necessary optimality conditions.
\begin{theorem}[\protect{\cite[Thm.~4.2]{CKK:2017}}]\label{thm:optimality}
  If $\eps>0$ and \cref{ass:data} holds, every local minimizer $\bar u\in BV(\omega_c)$ to \eqref{eq:problem} satisfies
  \begin{equation}\label{eq:optimality}
    -F_\eps'(\bar u)\in \alpha\,\partial G(\bar u) + \beta\,\partial \TV(\bar u)\subset L^\frac{2q}{2+q}(\omega_c),
  \end{equation}
  where $G$ and $\TV$ are considered as extended real-valued convex functionals on $L^\frac{2q}{q-2}(\omega_c)$.
\end{theorem}
Introducing explicit subgradients for the two subdifferentials, we obtain primal-dual optimality conditions.
\begin{corollary}[\protect{\cite[Cor.~4.3]{CKK:2017}}]\label{cor:optsys_pd}
  For any local minimizer $\bar u\in BV(\omega_c)$ to \eqref{eq:problem}, there exist $\bar q \in L^\frac{2q}{2+q}(\omega_c)$ and $\bar \xi \in L^{\frac{2q}{2+q}}(\omega_c)$ satisfying
  \begin{equation}\label{eq:optsys_pd}
    \left\{
      \begin{aligned}
        0 &= F_\eps'(\bar u) + \alpha\bar q + \beta \bar \xi,\\
        \bar q &\in \partial G(\bar u),\\
        \bar \xi &\in \partial \TV(\bar u).
      \end{aligned}
    \right.
  \end{equation}
\end{corollary}
These conditions can be further interpreted pointwise. First, using the characterization of \cref{lem:tracking_frechet}, we can identify the first term in the first equation with the $L^{1+\delta}(\omega_c)$ function given by
\begin{equation*}
  \left[\int_0^T\nabla y(u) \cdot \nabla p(u)\dt \right](x)\Phi_\eps'(\bar u(x)) \quad\text{for a.e. }x\in \omega_c.
\end{equation*}
Second, using the characterization of $\partial G$ from \cite[\S\,2]{CKK:2017}, we have that
\begin{equation*}
  \bar q(x) \in
  \begin{cases}
    \{-u_m\} & \bar u(x) < u_1,\\
    [-u_m, \frac12(u_1+u_2)] & \bar u(x) = u_1,\\
    \{\frac12(u_i+u_{i+1})\} & \bar u(x) \in (u_i,u_{i+1}), \quad\ 1\leq i<m,\\
    \left[\tfrac12(u_{i-1}+u_{i}),\tfrac12(u_{i}+u_{i+1})\right] & \bar u(x) = u_{i},\qquad\qquad 1\leq i<m,\\
    [ \frac12(u_{m-1}+u_m),u_m] & \bar u(x) = u_m,\\
    u_m & \bar u(x) > u_m.
  \end{cases}
\end{equation*}
The interpretation of the final term is more delicate. Informally, $\xi(x)$ corresponds to the mean curvature of $\bar u(x)$ (if $\bar u$ is smooth at $x$) or the signed normal to its jump set (if $\bar u$ has a jump discontinuity across a measurable curve of $d-1$-dimensional Hausdorff measure greater zero). This can be made more precise using the notion of the \emph{full trace} from \cite{BrediesHoller}; see also \cite{Chambolle:2015}.

\section{Numerical solution}\label{sec:numerical}

In this section, we address the numerical solution of \eqref{eq:problem} using a stabilized space-time finite element discretization for second-order hyperbolic equations \cite{Zlotnik1994} and a nonlinear primal-dual proximal splitting algorithm \cite{Valkonen:2014,CV:2017,CMV:2019}. Since we now consider a finite-dimensional optimization problem, we can include the constraint $u \in U$ directly via the multi-bang penalty instead of enforcing it inside the state equation. In this and the following section, we will therefore omit $\Phi_\eps$ from the discretized tracking term (and, with it, $\eps$ in general) and define the multi-bang penalty with $\mathrm{dom}\,g = [u_1,u_d]$ as in \cite{CK:2015}; see \eqref{eq:multi-bang-scalar} below.

\subsection{Discretization}\label{sec:discretization}

We consider a mesh $\mathcal{T}_h$ consisting of a finite set of triangles or tetrahedra $T$ with a mesh size $h$. Then we introduce the space $D_h\subset H^1(\Omega)\cap C(\overline \Omega)$ of linear finite elements based on the triangulation $\mathcal T_h$. A basis of this space is given by the standard hat functions $\phi_i$ associated with nodes $x_i$, $i=1,\ldots, N_h$, of the triangulation $\mathcal T_h$. Next we discretize the time interval $I$ uniformly by $0=t_0 < \cdots < t_{N_\tau}=T$ and grid size of $\tau$.
Similarly, we define the space $D_\tau\subset H^1(I)\cap C(\overline I)$ of piecewise linear and continuous functions with respect to these grids. Furthermore we consider the hat functions $e_i$, $i=0,\ldots,N_\tau$, with $e_i(t_l)=\delta_{il}$ which form a basis of $D_\tau$. We assume that $\omega_c$ can be represented by the triangulation $\mathcal T_h$ exactly and introduce the space
\begin{equation*}
  D_h^c=\{u_h\in D_h \mid \operatorname{supp}u_h\subseteq \overline{\omega_c}\}=\operatorname{span}\{\phi_i|_{\overline{\omega_c}}\mid x_i\in \overline{\omega_c}\}.
\end{equation*}
Moreover we introduce the space $C_h^c$ of piecewise constant functions on the triangles in $\omega_c$.
In the following we also identify $D_h^c$ with $\R^{N_c}$ for $\dim D_h^c=N_c$ and $C_h^c$ with $\R^{M_c}$ for $\dim C_h^c=M_c$. Finally we define $\vartheta:=(\tau,h) >0$ and introduce the stabilization parameter $\sigma \geq0$.
\begin{definition}
  We call $y_{\vartheta}\in D_{\vartheta}:=D_h \otimes D_\tau$ a discrete solution of \eqref{eq:weak} if $y_{\vartheta}$ satisfies
  \begin{multline}\label{eq:discreteSolutdef}
    \int_0^T-(\partial_t y_{\vartheta},\partial_t v)_{H}-(\sigma -\tfrac16)\tau^2(\hat Eu_h\nabla \partial_t y_{\vartheta},\nabla \partial_tv)_{\LL^2(\Omega)}+ (\hat Eu_h\nabla y_{\vartheta},\nabla v)_{\LL^2(\Omega)}\dt\\
    = (y_1,v(0))_{H} + \int_0^T(f,v)_{H}\dt
  \end{multline}
  for all $v \in D_{\vartheta}$ with $v(T)=0$ and initial condition $y_{\vartheta}(0)=S_0 y_0$ defined via
  \begin{equation*}
    (S_0 y_0,\varphi)_{H}=(y_0,\varphi)_{H}\quad \forall\varphi \in D_h.
  \end{equation*}
\end{definition}
This is a space-time finite element discretization with piecewise linear elements in space and in time. The additional $\sigma$-term in \eqref{eq:discreteSolutdef} serves as a stabilization term, which vanishes for $\vartheta\to 0$ and is connected to the error term in the trapezoidal rule for the time integral of the third bilinear form in \eqref{eq:discreteSolutdef}. 
The stability depends significantly on the value of $\sigma$ with the method being more stable for larger $\sigma$; e.g., for $\sigma\geq1/4$, the method is unconditionally stable and convergent while for $0\leq \sigma<1/4$, a CFL-like condition has to be satisfied to ensure stability as well as convergence; see \cite[Thms.~2.1, 3.1]{Zlotnik1994} for homogeneous Dirichlet boundary conditions. 
At the same time, \eqref{eq:discreteSolutdef} can be formulated as the following time-stepping scheme: Set $y_h^0=S_0y_0$ and
\begin{subequations}\label{eq:discreteSolutStep}
  \begin{align}
    &\begin{multlined}[t]
      \left(\frac{y_h^1-y_h^0}{\tau},\phi\right)_{L^2(\Omega)}+\tau(\hat Eu_h\nabla(\sigma y_h^1+(\tfrac12-\sigma)y_h^0),\nabla \phi)_{\LL^2(\Omega)}\\
      =(y_1,\phi)_{L^2(\Omega)}+\left(\int_0^{t_1}fe_0\dt,\phi\right)_{L^2(\Omega)},
    \end{multlined}
    \\
    &\begin{multlined}[t]
      \left(\frac{y_h^{i+1}-2y_h^i+y_h^{i-1}}{\tau},\phi\right)_{L^2(\Omega)}+\tau\left(\hat Eu_h\nabla(\sigma y_h^{i+1}+(1-2\sigma)y_h^i+\sigma y_h^{i-1}),\nabla \phi\right)_{\LL^2(\Omega)}\\
      =\left(\int_{t_{i-1}}^{t_{i+1}}fe_i\dt,\phi\right)_{L^2(\Omega)},\quad 1\leq i\leq N_\tau-1,
    \end{multlined}
  \end{align}
\end{subequations}
for all $\phi \in D_h$.
For $\sigma=1/4$, this method is equivalent to the implicit, unconditionally stable, and convergent Crank--Nicolson scheme, while for $\sigma=0$, the method is explicit if the spatial mass matrix is lumped.
The main benefit in our context is that this is an adjoint-consistent discretization and therefore can be used to obtain a conforming discretization of \eqref{eq:problem-full} in a straight-forward manner.

Next we introduce the discrete control-to-observation operator $S_\vartheta\colon U\cap D_h^c\to\calO$ defined by $u\mapsto By_\vartheta$ where $y_\vartheta$ is the solution of \eqref{eq:discreteSolutdef} for the coefficient $u\in U\cap D_h^c$. Let $\delta>0$. Then the implicit function theorem implies that $S_\vartheta$ is Fréchet differentiable on the open subset
\begin{equation*}
  \set{u \in L^\infty(\omega_c)}{-\inf_{x\in \omega_c} \hat u(x)+u_{\min} < u(x) < u_{\max}+\delta\quad\text{for a.e.}~x\in \omega_c}\cap D_h^c.
\end{equation*}
This set contains $U\cap D_h^c$. The implicit function theorem is applicable since the following linearized discrete state equation is well-posed in the variable $\delta y\in D_\vartheta$ for every $\delta u_h\in D_h^c$:
\begin{multline}\label{eq:discreteSolutdeflin}
  \int_0^T-(\partial_t \delta y_{\vartheta},\partial_tv)_{H}-(\sigma -\tfrac{1}{6})\tau^2(\hat Eu_h\nabla \partial_t \delta y_{\vartheta},\nabla \partial_tv)_{\LL^2(\Omega)}+ (\hat Eu_h\nabla \delta y_{\vartheta},\nabla v)_{\LL^2(\Omega)}\dt\\
  =\int_0^T(\sigma -\tfrac{1}{6})\tau^2(\hat E\delta u_h\nabla \partial_t y_{\vartheta},\nabla \partial_tv)_{\LL^2(\Omega)}- (\hat E\delta u_h\nabla y_{\vartheta},\nabla v)_{\LL^2(\Omega)}\dt
\end{multline}
for all $v \in D_{\vartheta}$ with $v(T)=0$ and initial condition $\delta y_{\vartheta}(0)=0$ as well as $y_\vartheta=S_\vartheta(u)$.
Thus the derivative of $S_\vartheta$ at $u$ is given by $S'(u_h)\colon D_h^c\to \calO,~\delta u_h\mapsto B\delta y_\vartheta$ where $\delta y_\vartheta$ solves \eqref{eq:discreteSolutdeflin} for $\delta u_h$.
Its adjoint (with respect to the $L^2(Q)$ and $\calO$ inner product) is given by
\begin{equation*}
  S'(u_h)^\ast\colon \calO\to C_h^c\, , \qquad
  o\mapsto
  \left.
  \left(
    \int_0^T(\tfrac{1}{6}-\sigma)\tau^2\nabla \partial_t y_{\vartheta}\cdot\nabla \partial_tp_\vartheta+\nabla y_{\vartheta}\cdot\nabla p_\vartheta\dt
  \right)
  \right|_{\omega_c}
\end{equation*}
where $y_\vartheta=S_\vartheta(u)$ and $p_{\vartheta}$ solves the discrete adjoint equation
\begin{equation}\label{eq:discreteadjoint}
  \int_0^T-(\partial_t v,\partial_tp_\vartheta)_{H}-(\sigma -\tfrac{1}{6})\tau^2(\hat Eu_h\nabla \partial_t v,\nabla \partial_tp_\vartheta)_{\LL^2(\Omega)}+ (\hat Eu_h\nabla v,\nabla p_\vartheta)_{\LL^2(\Omega)}\dt\\
  = \int_0^T(B^\ast o,v)_{H}\dt
\end{equation}
for all $v \in D_{\vartheta}$ with $v(0)=0$ and initial condition $p_{\vartheta}(T)=0$, which can be formulated as a time-stepping scheme similar to \eqref{eq:discreteSolutStep}.

We now introduce the variables $y_h$ and $p_h$ defined by
\begin{equation*}
  y_\vartheta(t,x)=\sum_{i=0}^{N_\tau}y_h^i(x)e_i(t),\quad p_\vartheta(t,x)=\sum_{i=1}^{N_\tau}p_h^i(x)e_i(t).
\end{equation*}
Thus $(S_\vartheta'(u_h))^\ast o$ with $o\in \calO$ has the representation
\begin{equation}\label{eq:disrete-adjoint-operator}
  (S_\vartheta'(u_h))^\ast o= \nabla y_h^\top K\nabla p_h,
\end{equation}
where $K$ is given by
\begin{equation*}
  K:=(\tfrac 16-\sigma)\tau^2A_\tau+M_\tau=\tau
  \begin{pmatrix}
    \tfrac 12-\sigma& \sigma&&\\
    \sigma& 1-2\sigma& \sigma&\\
    &\ddots&\ddots&\ddots\\
    &&&\qquad\sigma\\
    &&\sigma&\tfrac12-\sigma
  \end{pmatrix}
\end{equation*}
with the temporal mass matrix $M_\tau$ and stiffness matrix $A_\tau$. For $\sigma=0$, $K$ is a diagonal matrix.

\bigskip

We now address the discretization of the control costs in the optimization \eqref{eq:problem-full}. Since a function $u_h\in D_h^c$ is an element of $H^1(\omega_c)$ and thus its weak derivatives are piecewise constant on the triangulation of $\omega_c$, we have
\begin{equation*}
  TV(u_h)=\|\nabla u_h\|_{L^1(\omega_c)}=\sum_{K\in \mathcal T_h\cap \omega_c}|(\nabla u_h)|_K|_2.
\end{equation*}
We furthermore approximate the integral in definition of $G$ by the trapezoidal rule to obtain the discrete multi-bang penalty $G_h\colon D_h^c \to \Rbar$.
Using these definitions, we obtain the fully discrete optimization problem
\begin{equation}\label{eq:discrete-problem-func}
  \min_{u_h\in D_h^c\cap U}\frac 12 \|S_\vartheta(u_h)-y_d\|_{\calO}^2+ \alpha G_h(u_h) + \beta TV(u_h).
\end{equation}

Note that although \eqref{eq:discrete-problem-func} is discrete, it is still formulate in function spaces. To apply a minimization algorithm, we now reformulate it in terms of the coefficient vectors for the finite-dimensional functions.
First, $D_h^c\cap U$ can be identified with the set
\begin{equation*}
  U_h=\{\textbf u\in \R^{N_c}|~0\leq \textbf u_i\leq u_{\max}~\forall i=1,\ldots,N_c\}
\end{equation*}
through $u_h=\sum_{x_i\in \overline{\omega_c}}\textbf u_ie_i$.
Next we introduce the finite-dimensional subspace $\calO_h=B(D_\vartheta)$ of $\calO$ and the discrete control-to-observation on $U_h$ by $\textbf S_\vartheta\colon \R^{N_c}\to \R^{N_o}$ with $N_o=\dim(\calO_h)$ defined by $S_\vartheta$. Moreover we define the matrix $M_{\calO}\in \R^{N_o\times N_o}$ by the mapping $(o_1,o_2)\mapsto (o_1,o_2)_{\calO}$ for $o_1,o_2\in \calO_h$. Thus the inner product and the norm of $\calO$ in $\calO_h$ can be identified with $(\textbf o_1,\textbf o_2)_{\calO_h}=\textbf o_1^\top M_{\calO}\textbf o_2$ and $\|\textbf o\|_{\calO_h}=((\textbf o,\textbf o)_{\calO_h})^{1/2}$ for $\textbf o,~\textbf o_1,~\textbf o_2\in \R^{N_o}$. We denote the orthogonal projection onto $\calO_h$ by $\pi_{\calO}$. The operator $B$ restricted to $D_\vartheta$ can be identified with a matrix $B_h\in \R^{N_o\times N_{\vartheta}}$ for $N_{\vartheta}=\dim(D_\vartheta)$. Thus $B^\ast$ can be identified with $B_h^\top$.
With these identifications, the adjoint operator $S_\vartheta'(u_h)^\ast$ for $u_h\in U\cap D_h^c$ acting on $\calO_h$ can similarly identified with $\textbf S_\vartheta'(\textbf u)^\ast\colon \R^{N_o}\to \R^{M_c}$ for some $\textbf u\in U_h$.
Moreover we define the matrix $A_h\in \R^{dM_c\times N_c}$ representing the bilinear form $(\nabla u_h,\xi_h)_{L^2(\Omega)^d}$ for $u_h\in D_h^c$ and $\xi_h\in (C_h^c)^d$. Thus we have
\begin{equation*}
  TV(u_h)=\sum_{K\in \mathcal T_h\cap \omega_c}|(\nabla u_h)|_K|_2=\sum_{K\in \mathcal T_h\cap \omega_c}|(A_h\textbf u)|_K|_2 =: \norm{|A_h \textbf u|_2}_1=:\textbf{TV}_h (\textbf u).
\end{equation*}
Finally, the trapezoidal rule in the definition of $G_h$ can be expressed in the form of a mass lumping scheme, i.e.,
\begin{equation*}
  G_h(u_h)=\sum_{x_i\in \overline{\omega_c}}d_ig(\textbf u_i)=: \textbf G_h(\textbf u)
\end{equation*}
where
\begin{equation}\label{eq:multi-bang-scalar}
  g(t) = \begin{cases}
    \infty & t < u_1,\\
    \frac12 \left((u_{i}+u_{i+1})t - u_iu_{i+1}\right) & t \in [u_i,u_{i+1}],\quad 1\leq i <m,\\
    \infty & t > u_m,
  \end{cases}
\end{equation}
is the scalar multi-bang penalty including the box constraints from \cite{CK:2015} and
$d_i=\int_{\omega_c}\phi_i\dx$ are the diagonal entries of the lumped mass matrix; see \cite{CHW:2012,Pieper,Trautmann,RoeschWachsmuth:2017}.

Using these notations, we can write \eqref{eq:discrete-problem-func} equivalently in the form
\begin{equation}\label{eq:discrete-problem}
  \min_{\textbf u\in U_h}\frac 12 \|\textbf S_\vartheta(\textbf u)-\pi_{\calO}y_d\|_{\calO_h}^2+ \alpha \textbf G_h(\textbf u) + \beta \textbf{TV}_h(\textbf u).
\end{equation}

\subsection{Primal-dual proximal splitting}\label{sec:pdps}

To solve \eqref{eq:discrete-problem}, we extend the approach in \cite{TramDo} by applying the nonlinear primal-dual proximal splitting method from \cite{Valkonen:2014,CV:2017,CMV:2019} together with a lifting trick. For this purpose, we write \eqref{eq:discrete-problem} (omitting the bold notation and the subscripts denoting vectors and discretizations from now on and assuming that $y_d\in \calO_h$) as
\begin{equation*}
  \min_{u\in U} \frac12 \norm{S(u) - y_d}_{\calO}^2 + \beta \norm{|A_hu|_2}_1 + \alpha G(u).
\end{equation*}
Setting
\begin{align*}
  &\mathcal F: \R^{N_o} \times \R^{2M_c} \to \R, &&(y,\psi)\mapsto \frac12\norm{y-y_d}_{\calO}^2 + \beta\norm{|\psi|_2}_1,\\
  &K:U\to (\R^{N_o} \times \R^{2M_c}), && u\mapsto (S(u),A_h u),
\end{align*}
we can apply the nonlinear primal-dual proximal splitting algorithm
\begin{equation}\label{eq:nlpdps}
  \left\{
    \begin{aligned}
      u^{k+1} &= \prox_{\gamma_G \alpha G_h}\left(u^k-\gamma_G K'(u^k)^* \xi^k\right)\\
      \hat u^{k+1} &= 2 u^{k+1} - u^k\\
      \xi^{k+1} &= \prox_{\gamma_F \calF^*}\left(\xi^k + \gamma_F K(\hat u^{k+1})\right)
    \end{aligned}
  \right.
\end{equation}
for step sizes $\gamma_F,\gamma_G>0$ satisfying $\gamma_F\gamma_G\norm{K'(u)^*}<1$. Convergence can be guaranteed under a second-order type condition for $K$ and possibly further restrictions on the step sizes, whose (very technical) verification is outside the scope of this work.
Instead, we restrict the discussion here on deriving the explicit form of \eqref{eq:nlpdps} in the present setting.

First, we endow $(\R^{N_o}\times \R^{2M_c})$ with the sum of the inner product induced by $M_\calO$ (for $\R^{N_o}$) and the Euclidean inner product (for $\R^{2M_c}$). With respect to this inner product, we obtain the adjoint Fréchet derivative
\begin{equation*}
  K'(u)^*(r,\xi) = S'(u)^*(r) + A_h^T\psi,
\end{equation*}
where $S'(u)^*$ is the fully discrete operator corresponding to \eqref{eq:disrete-adjoint-operator} with right-hand side $r\in \calO$ for the adjoint equation.

The proximal point mapping for the (scaled) multi-bang penalty can be obtained by straightforward calculation based on a case differentiation in the definition of the subdifferential, see \cite[Prop.~3.6]{TramDo}; for the sake of completeness, we give the short derivation here in full.
By the definition of the proximal mapping, $w=\prox_{\gamma g}(v) = (\Id + \gamma \partial g)^{-1}(v)$ holds for any $v\in \R$ if and only if $v \in \{w\} + \gamma \partial g(w)$.
Recalling from \cite[\S\,2]{CD:2017} that
\begin{equation}
  \label{eq:mb-subdiff-h}
  \partial g(v) =
  \begin{cases}
    (-\infty,\tfrac12(u_1+u_2)] & v = u_1,\\
    \{\tfrac12(u_{i}+u_{i+1})\} & v \in (u_{i},u_{i+1}), \quad 1 \leq i < m,\\
    [\tfrac12(u_{i-1}+u_i), \tfrac12(u_i+u_{i+1}] & v = u_i, \quad 1<i<m,\\
    [\tfrac12(u_{m-1}+u_m),\infty) & v = u_m,\\
    \emptyset & \text{otherwise},
  \end{cases}
\end{equation}
we now distinguish the following cases for $w$:
\begin{enumerate}
  \item $w=u_1$: In this case,
    \begin{equation*}
      v \in \{w\} + \gamma \left(-\infty,\tfrac12(u_1+u_2)\right] = \left(-\infty,(1+\tfrac{\gamma}2) u_1 + \tfrac{\gamma}2 u_2\right].
    \end{equation*}
  \item $w\in (u_i,u_{i+1})$ for $1\leq i< m$: In this case,
    \begin{equation*}
      v \in \{w\} + \gamma \{\tfrac12(u_i + u_{i+1})\},
    \end{equation*}
    which first can be solved for $w$ to yield
    \begin{equation*}
      w = v - \tfrac{\gamma}2 (u_i+u_{i+1});
    \end{equation*}
    inserting this into $w\in (u_i,u_{i+1})$ and simplifying then gives
    \begin{equation*}
      v \in \left((1+\tfrac{\gamma}2)u_i + \tfrac{\gamma}2 u_{i+1}, \tfrac{\gamma}2 u_i + (1+\tfrac{\gamma}2)u_{i+1}\right).
    \end{equation*}
  \item $w = u_i$, $1<i<m$: Proceeding as in the first case, we obtain
    \begin{equation*}
      v \in \left[\tfrac{\gamma}2 u_{i-1},(1+\tfrac{\gamma}2)u_i,(1+\tfrac{\gamma}2) u_i + \tfrac{\gamma}2 u_{i+1}\right].
    \end{equation*}
  \item $w=u_m$: Similarly, this implies that
    \begin{equation*}
      v \in \left[\tfrac{\gamma}2 u_{m-1},(1+\tfrac{\gamma}2)u_m,\infty\right).
    \end{equation*}
\end{enumerate}
Since this is a complete and disjoint case distinction for $v\in \R$, we obtain the proximal mapping for the scalar penalty $g$; see \cref{fig:mbprox}.
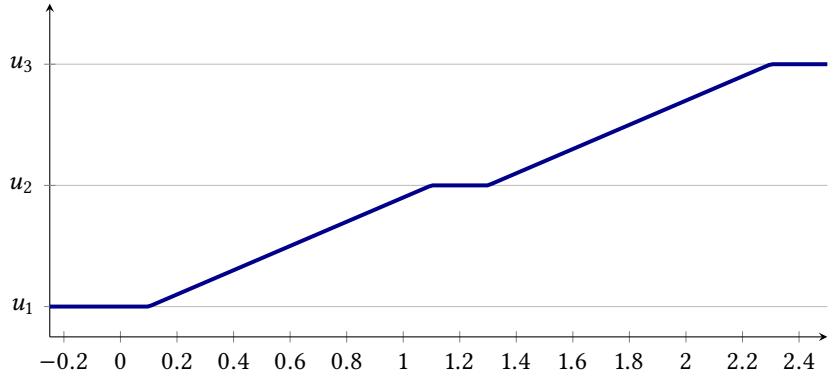
\begin{figure}
  \centering
  \begin{tikzpicture}
    \begin{axis}[%
      width=0.75\linewidth,
      height=6cm,
      xmin=-0.25,
      xmax=2.5,
      ymin=-0.25,
      ymax=2.5,
      ytick=\empty,
      extra y tick style={grid=major},
      extra y ticks={0,1,2},
      extra y tick labels={$u_1$,$u_2$,$u_3$},
      axis y line=left,
      axis x line=bottom,
      ]
      \pgfmathsetmacro{\ua}{0}
      \pgfmathsetmacro{\ub}{1}
      \pgfmathsetmacro{\uc}{2}
      \pgfmathsetmacro{\at}{0.2}
      \addplot[%
        domain=-0.25:2.5,samples=200,
        color=DarkBlue,solid,line width=1.5pt,
        ]
        {
          (x<=\ua)*\ua + %
          and(x>(1+\at/2)*\ua+\at/2*\ub,x<(1+\at/2)*\ub+\at/2*\ua)*(x-\at/2*(\ua+\ub)) + %
          and(x>=(1+\at/2)*\ub+\at/2*\ua,x<=(1+\at/2)*\ub+\at/2*\uc)*\ub + %
          and(x>(1+\at/2)*\ub+\at/2*\uc,x<(1+\at/2)*\uc+\at/2*\ub)*(x-\at/2*(\ub+\uc)) + %
          (x>=(1+\at/2)*\uc+\at/2*\ub)*\uc
        };
    \end{axis}
  \end{tikzpicture}
  \caption{Illustration of the proximal point mapping $\prox_{\gamma g}$ of the scalar multi-bang penalty $g$ for $(u_1,u_2,u_3)$ = $(0,1,2)$ and $\gamma=0.2$}
  \label{fig:mbprox}
\end{figure}
By a standard argument, the proximal point mapping for $G_h$ is thus given componentwise by
\begin{equation*}
  [\prox_{\gamma \alpha G_h}(v)]_j =
  \begin{cases}
    v_j-\frac{\alpha\gamma}2 (u_{i} + u_{i-1}) & \text{if } v_j \in \left(\left(1+\tfrac{\alpha\gamma}2\right)u_{i-1}+\tfrac{\alpha\gamma}2 u_{i},\left(1+\tfrac{\alpha\gamma}2\right)u_{i}+\tfrac{\alpha\gamma}2 u_{i-1}\right),\\
    u_i &\text{if } v_j\in \left[\left(1+\tfrac{\alpha\gamma}2\right)u_{i}+\tfrac{\alpha\gamma}2 u_{i-1},\left(1+\tfrac{\alpha\gamma}2\right)u_{i}+\tfrac{\alpha\gamma}2 u_{i+1}\right],
  \end{cases}
\end{equation*}
where we have set $u_{0}=-\infty$ and $u_{m+1}=\infty$ to avoid the need for further case distinctions.
(Note that we compute the proximal mapping with respect to the inner product induced by the lumped mass matrix such that the weight $d_i$ cancels.)

Finally, for $\calF$, we first compute the Fenchel conjugate on $\R^{N_o}\times \R^{2M_c}$ (with respect to the same inner product as above) as
\begin{equation*}
  \calF^* (r,\psi)
  = \frac{1}{2}\norm{r}_{\calO}^2 + (r,y^d)_{\calO} + \delta_{\{|\psi|_2\leq \beta\}}(\psi)
\end{equation*}
to obtain the proximal point mapping (again, with respect to this inner product)
\begin{equation*}
  \prox_{\gamma \calF^*}(r,\psi) =
  \begin{pmatrix}
    \frac{1}{1+\gamma} (r - \gamma y^d)\\
    \proj_{\{|\psi|_2\leq \beta\}}(\psi)
  \end{pmatrix},
\end{equation*}
where the projection can be computed elementwise for each $K\in \calT\cap \omega_c$ as
\begin{equation*}
  [ \proj_{\{|\psi|_2\leq \beta\}}(\psi)]_K = \frac{\beta \psi}{\max\{\beta,|\psi|_2\}}.
\end{equation*}
With these, \eqref{eq:nlpdps} becomes the following explicit algorithm:
\begin{align*}
  u^{k+1} &= \prox_{\gamma_G\alpha G_h} \left(u^k-\gamma_G S'(u^k)^*r^k - \gamma_G A_h^T\psi^k\right),\\
  \hat u^{k+1} &= 2 u^{k+1} - u^k,\\
  y^{k+1} &= S(\hat u^{k+1}),\\
  r^{k+1} &= \frac{1}{1+\gamma_F} \left(r^k + \gamma_F(y^{k+1}- y^d)\right),\\
  q^{k+1} &= \psi^k + \gamma_F A_h\bar u^{k+1},\\
  \psi^{k+1} &= \proj_{\{|\psi|_2\leq \beta\}}(\psi).
\end{align*}
Note that this requires two solutions of the forward wave equation (as well as one solution of the adjoint equation) in each iteration, since $y^{k+1}$ is based on the extrapolated vector $\bar u^{k+1}$, while the state vector $y$ required for the computation of $S'(u^{k+1})^*r^{k+1}$ in the following iteration is based on the original update $u^{k+1}$.

The iteration is terminated based on the residual norm in an equivalent reformulation of the optimality conditions for \eqref{eq:discrete-problem-func}. Combining the approach of \cref{sec:existence} with standard results from convex analysis (see, e.g, \cite{Bauschke:2011,Clason:2017}), any local minimizer $\bar u$ of \eqref{eq:discrete-problem-func} together with the corresponding Lagrange multiplier $\bar \psi$ and the residual $\bar r := S(\bar u)-y_d$ can be shown to satisfy
\begin{equation*}
  \left\{
    \begin{aligned}
      \bar u &= \prox_{\gamma_G\alpha G_h} \left(\bar u-\gamma_G S'(\bar u)^*\bar r - \gamma_G A_h^T\bar \psi\right),\\
      \bar r &= \frac{1}{1+\gamma_F} \left(\bar r + \gamma_F(S(\bar u) - y_d)\right),\\
      \bar \psi &= \proj_{\{|\psi|_2\leq \beta\}}(\bar \psi).
    \end{aligned}
  \right.
\end{equation*}
For the first equation, which holds in $U_h$, we measure the residual in the discrete norm induced by the lumped mass matrix as in the definition of $G_h$. The second equation holds in $\calO_h$, and hence we measure the residual in the norm induced by the corresponding mass matrix $M_\calO$. Finally, the last equation holds in $\R^{2N_c}$ so we use the standard Euclidean norm. The iteration is terminated once the sum of these residuals drops below a given tolerance. For the implementation, note that the residual in the first equation for $(u^k,r^k,\psi^k)$ reduces to $u^{k}-u^{k+1}$. On the other hand, the residual in the second equation requires an additional solution of the state equation since here $S$ is applied to $u^k$ instead of the extrapolated $\bar u^k$. In practice, we thus do not evaluate the stopping criterion in every iteration.

\section{Numerical examples}\label{sec:examples}

We now illustrate the above presented approach with two numerical examples. The first is a transmission problem (where waves produced by external forcing pass through the control domain before being observed) loosely motivated by acoustic tomography. The second is a reflection problem (where only reflected, not transmitted, waves are observed) that more closely models seismic inversion.
The implementation in Python (using DOLFIN \cite{LoggWells2010a,LoggWellsEtAl2012a}, which is part of the open-source computing platform FEniCS \cite{AlnaesBlechta2015a,LoggMardalEtAl2012a}) used to generate the following results can be downloaded from \url{https://github.com/clason/tvwavecontrol}.

\subsection{A model acoustic tomography problem}\label{sec:ex-transmission}

For the first example, we take $\Omega = (-1,1)\times (-1,2)$ and $T=3$ and define the control and observation domains
\begin{equation*}
  \omega_c := \{ (x_1,x_2)\in \Omega \mid x_2 \in (0,1) \},\qquad
  \omega_o := \{ (x_1,x_2)\in \Omega \mid x_2 \in (1,2) \};
\end{equation*}
correspondingly, the observation operator is taken as the restriction operator $By := y|_{\calO}$ to the observation space $\calO := (0,T)\times \omega_o$.
The initial conditions are chosen as $(y_0,y_1) = (0,0)$, thus satisfying \cref{ass:data}.
We now aim to recover a piecewise constant coefficient $u_e$ with $u_e(x) \in \{1.0,1.1,1.2,1.3,1.4\}$ almost everywhere; see \cref{fig:transmission-exact}.
\begin{figure}
  \centering
  \includegraphics[width=0.45\linewidth]{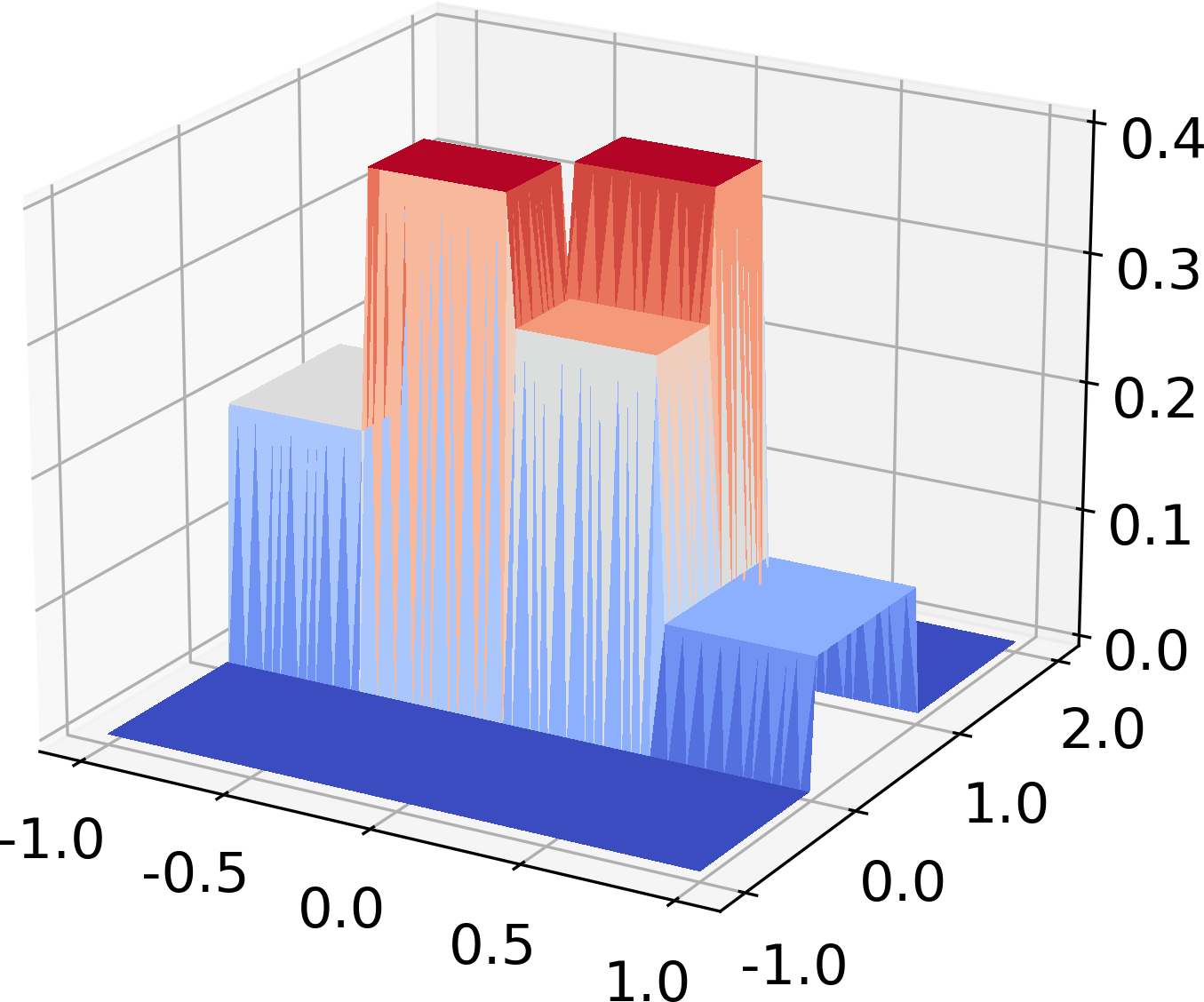}
  \hfill
  \includegraphics[width=0.45\linewidth]{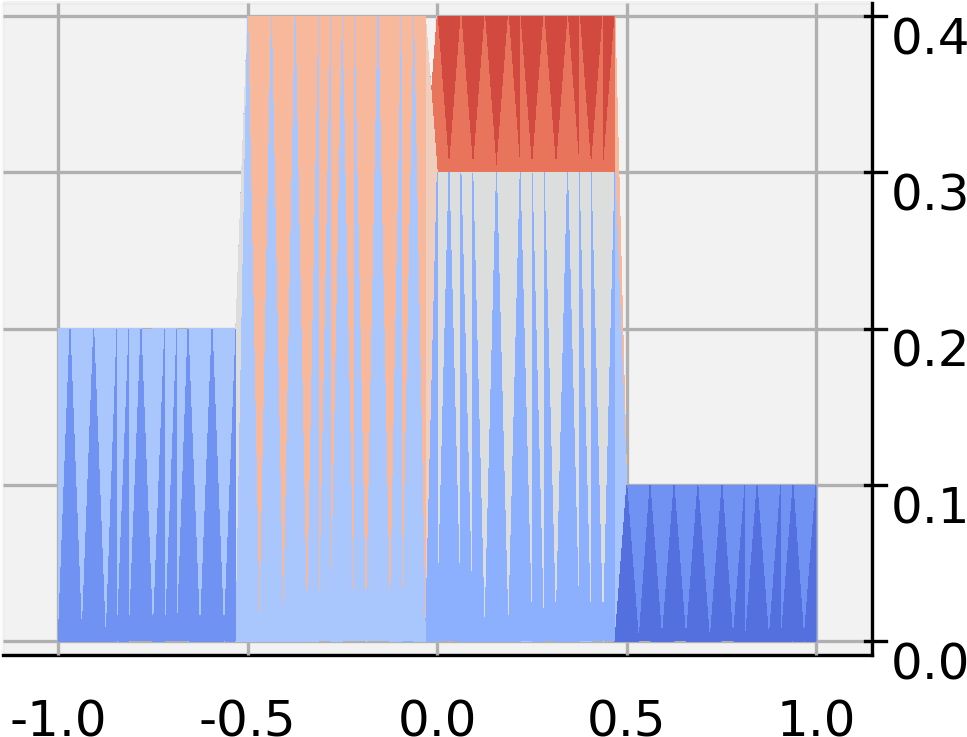}
  \caption{transmission example, exact coefficient $u_e$}
  \label{fig:transmission-exact}
\end{figure}
Accordingly, we set $\hat u \equiv 1.0$ and $u_i = (i-1)/10$, $i=1,\dots 5$ from noisy observations of the state in $\calO$. These observations are generated using a source term $f$ that is constructed as a linear combination of point sources which act as Ricker wavelets in time, i.e.,
\begin{equation*}
  f(x,t) := \sum_{i=-9}^9 (\delta_{(i/10,-9/10)}(x) + \delta_{(0.05+i/10,-8/10)}(x))
  2(1-2(5\pi(t-0.1))^2)e^{-(5\pi(t-0.1))^2}.
\end{equation*}
(The number and location of source points as well as the amplitude and frequency of the wavelet are chosen such as to obtain a complex enough wave pattern to recover the lateral and depth-wise variations in the coefficient.)
The discretization is performed using $64$ nodes in each space direction and $128$ nodes in time, corresponding to $h\approx 0.056$ and $\tau \approx 0.23$. The stabilization constant is set to $\sigma = 1/4$.
The discretized exact solution is then perturbed componentwise by $10\%$ relative Gaussian noise, i.e., we take
\begin{equation*}
  y_d = B(y(u)) + 0.1\norm{B(y(u))}_{\infty}\xi,
\end{equation*}
where $\xi$ is a vector of independent normally distributed random components with mean $0$ and variance $1$.

We now compute the reconstruction using the algorithm described in \cref{sec:pdps}, comparing the effects of the total variation and the multi-bang penalty by taking $\alpha \in \{0,10^{-5}\}$ and $\beta \in \{0,10^{-4}\}$. In each case, we set the step sizes to $\gamma_F = 10^{-1}$ and $\gamma_G = 10^3$ and terminate when the residual norms (evaluated every $10$ iterations) drop below $10^{-6}$.
Again, these parameters are chosen to achieve a reasonable reconstruction in as few iterations as possible. (A proper parameter choice rule depending on the noise level and the discretization is left for future work.)
The results can be seen in \cref{fig:transmission-reconstruction}. The case of pure multi-bang regularization ($\alpha = 10^{-5}$ and $\beta = 0$, $3680$ iterations); see \cref{fig:transmission-reconstruction:mb}) shows that indeed $u(x)\in \{u_1,\dots,u_5\}$ almost everywhere; however, there is a clear lack of regularity of the reconstruction, which is not surprising as the original function-space problem is not well-posed for $\beta=0$. In contrast, the reconstruction case of pure $\TV$ regularization ($\alpha = 0$ and $\beta = 10^{-4}$, $1100$ iterations; see \cref{fig:transmission-reconstruction:tv}) shows a much more regular reconstruction that is constant on large regions; however, these constants are not necessarily from the admissible set $\{u_1,\dots,u_5\}$. Finally, combining both multi-bang and total variation regularization ($\alpha = 10^{-5}$ and $\beta = 10^{-4}$, $600$ iterations; see \cref{fig:transmission-reconstruction:tvmb}) allows recovering more admissible values at the price of penalizing the magnitude of the coefficient value, which prevents the largest value $u_5=0.4$ from being attained. It is also noteworthy that in this case the tolerance for the residual norm is reached after significantly fewer iterations.
\begin{figure}[p]
  \centering
  \begin{subfigure}{\linewidth}
    \centering
    \includegraphics[width=0.45\linewidth]{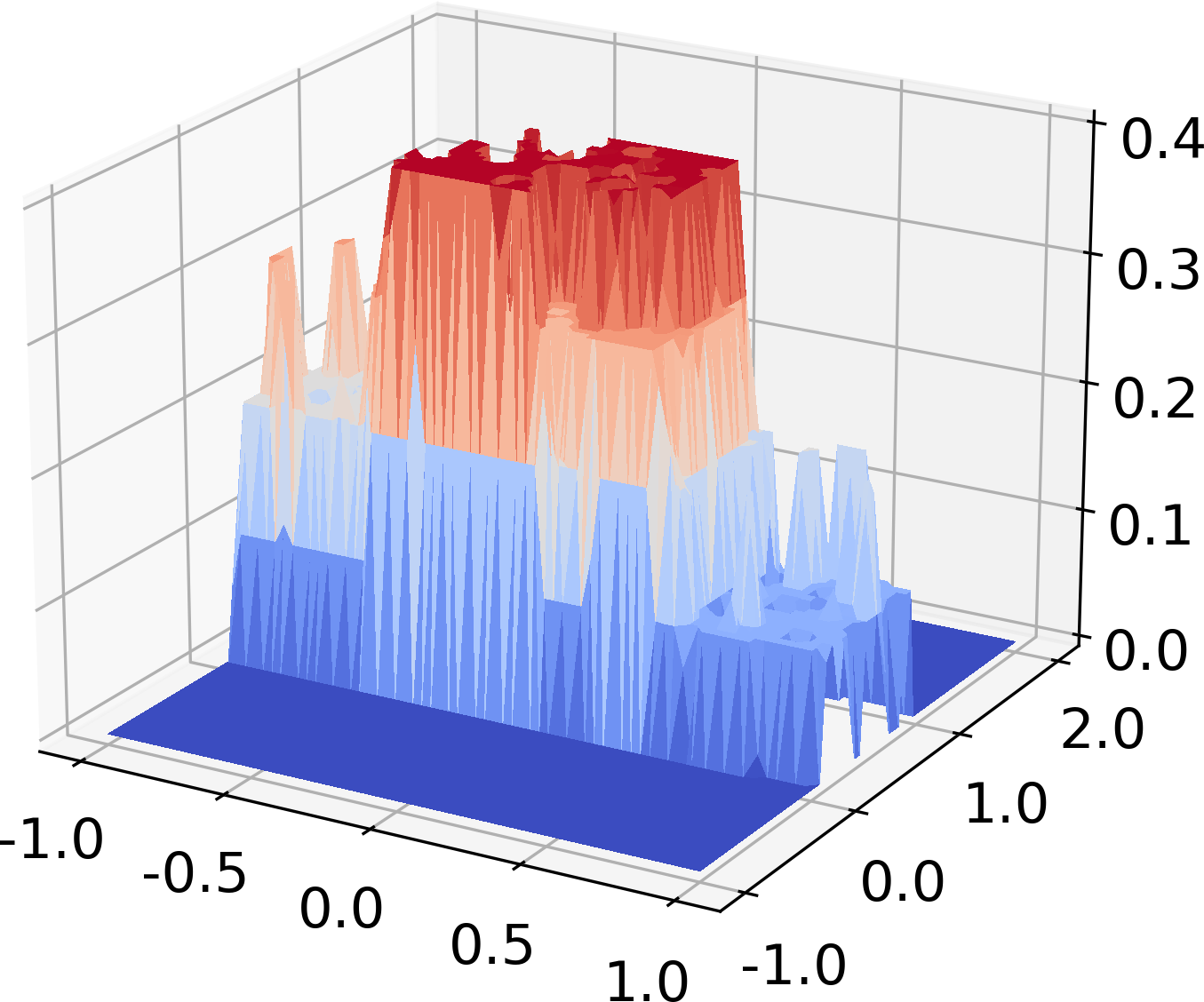}
    \hfill
    \includegraphics[width=0.45\linewidth]{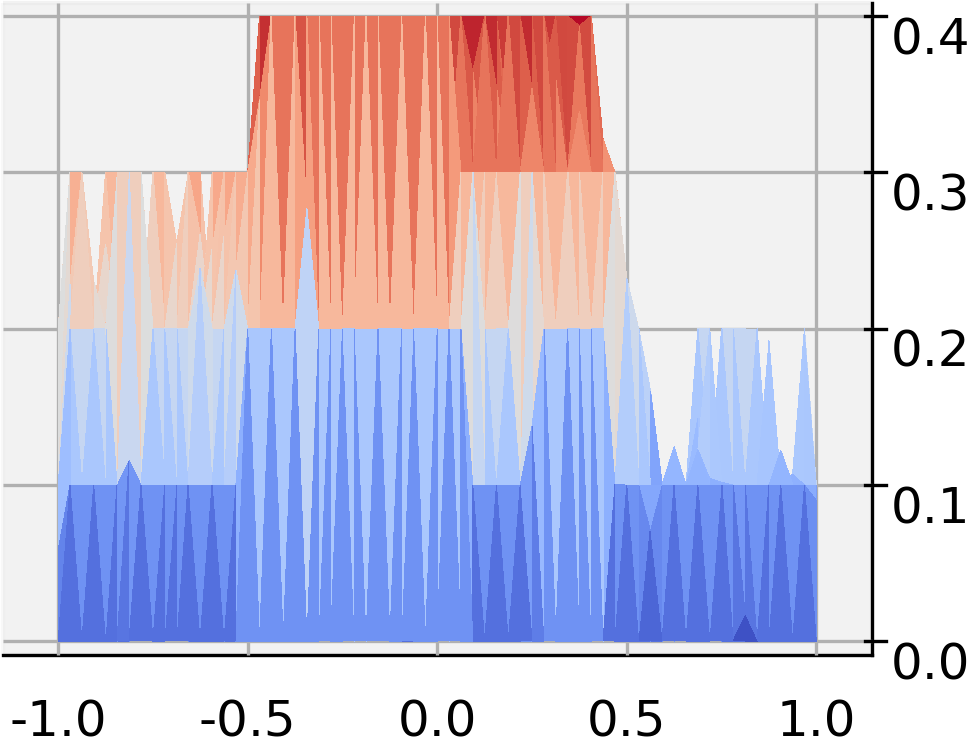}
    \caption{$\alpha = 10^{-5}$, $\beta = 0$, $3680$ iterations}
    \label{fig:transmission-reconstruction:mb}
  \end{subfigure}

  \begin{subfigure}{\linewidth}
    \centering
    \includegraphics[width=0.45\linewidth]{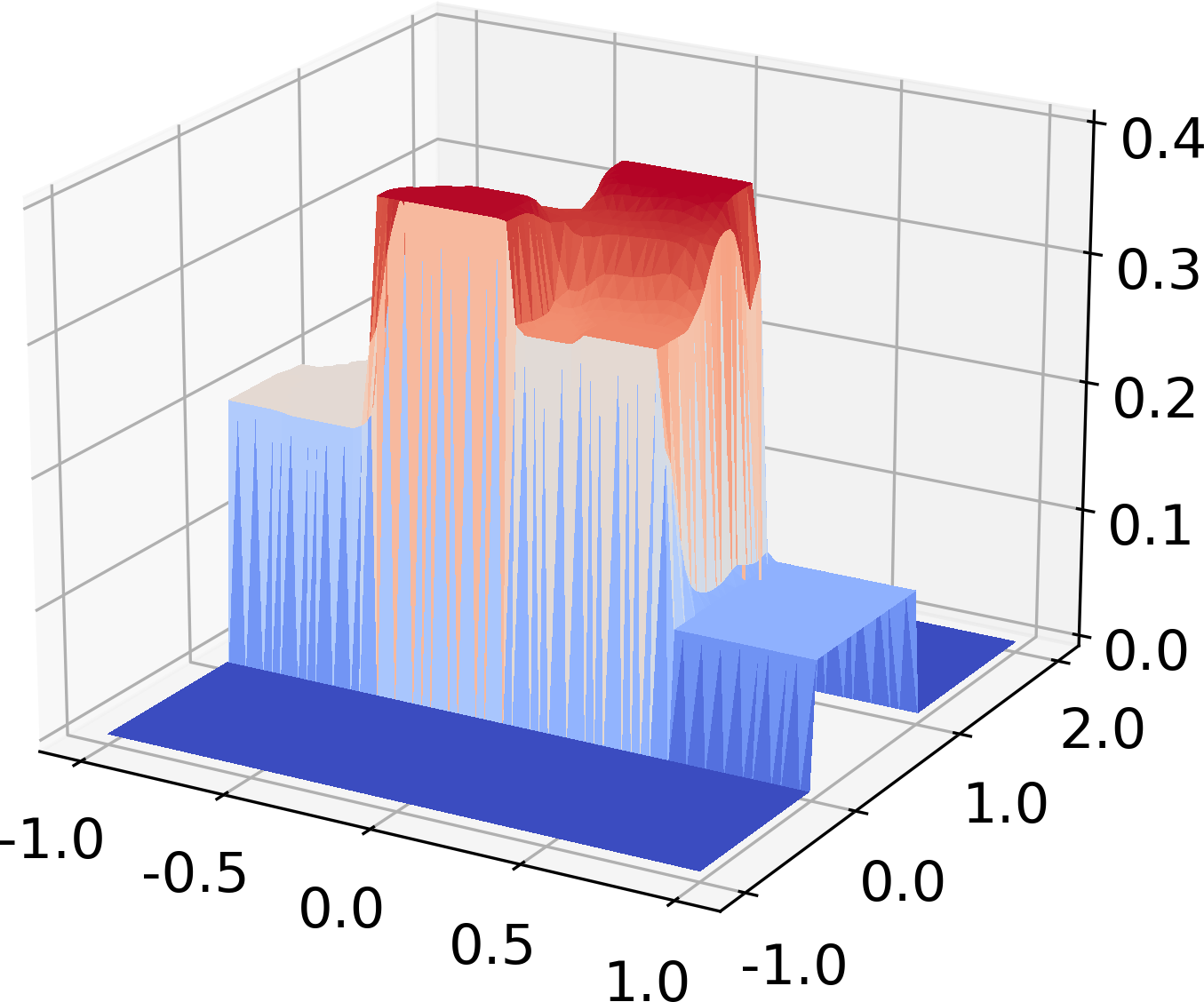}
    \hfill
    \includegraphics[width=0.45\linewidth]{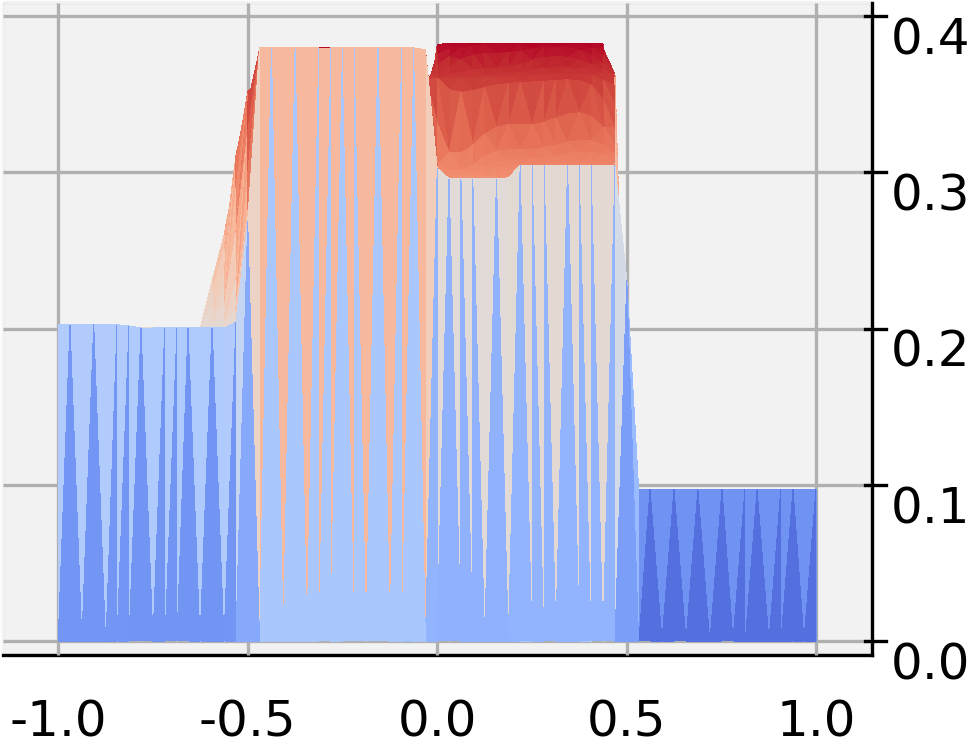}
    \caption{$\alpha = 0$, $\beta = 10^{-4}$, $1100$ iterations}
    \label{fig:transmission-reconstruction:tv}
  \end{subfigure}

  \begin{subfigure}{\linewidth}
    \centering
    \includegraphics[width=0.45\linewidth]{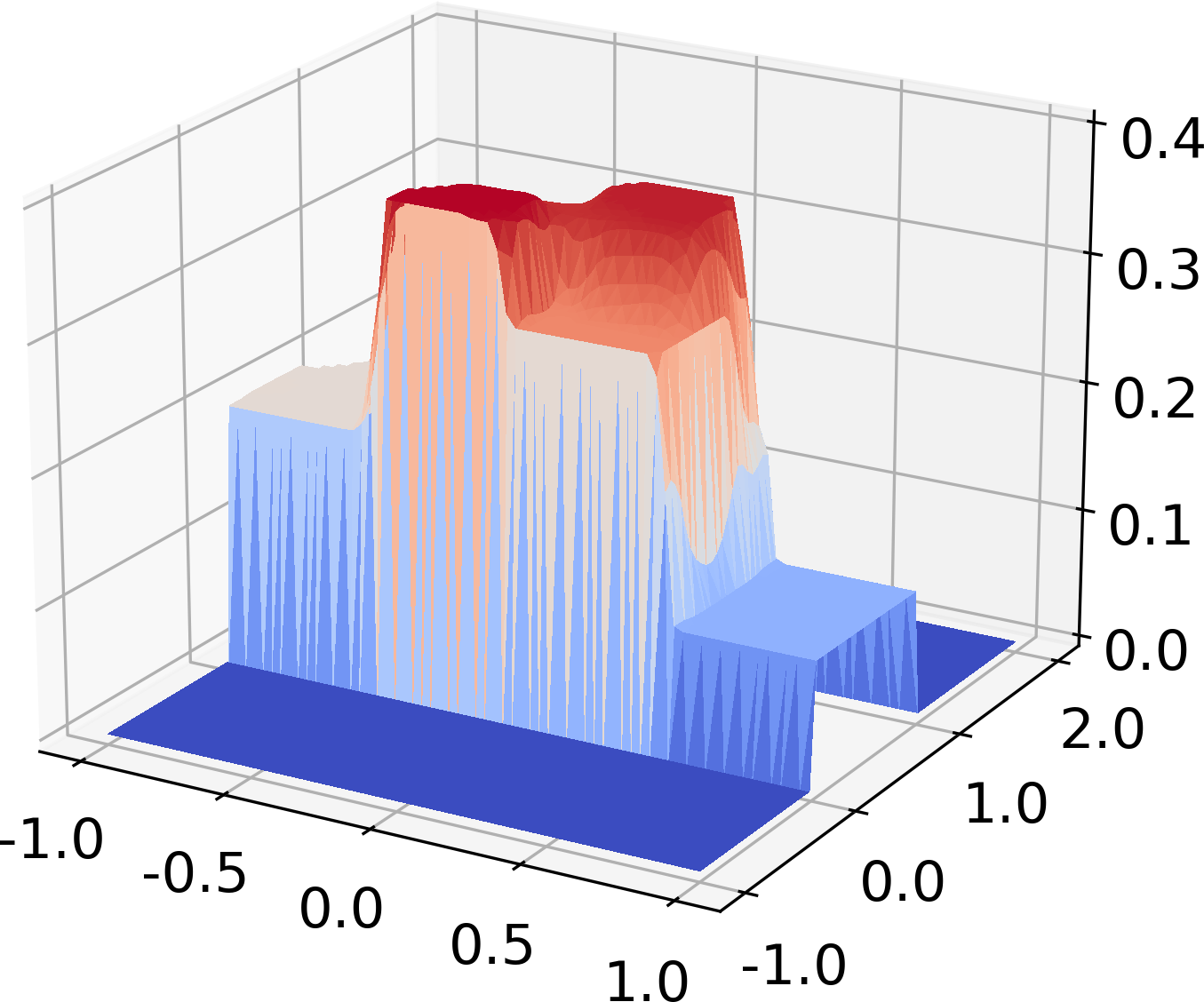}
    \hfill
    \includegraphics[width=0.45\linewidth]{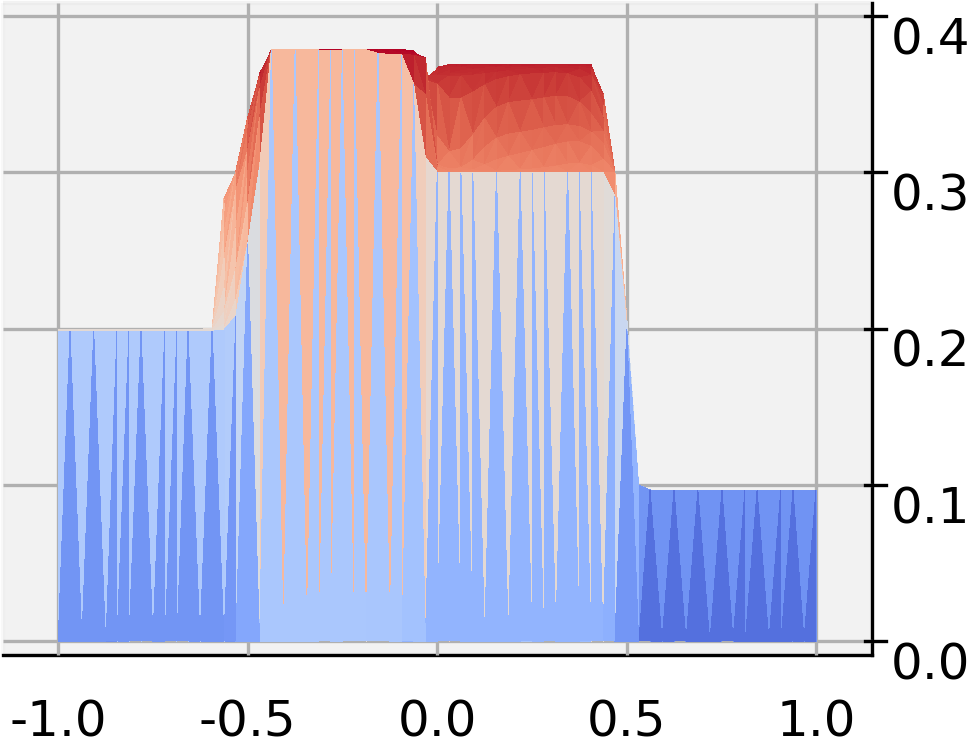}
    \caption{$\alpha = 10^{-5}$, $\beta = 10^{-4}$, $600$ iterations}
    \label{fig:transmission-reconstruction:tvmb}
  \end{subfigure}
  \caption{transmission example, effect of multi-bang penalty ($\alpha$) and total variation penalty ($\beta$) on the reconstruction}
  \label{fig:transmission-reconstruction}
\end{figure}

To illustrate the effects of variation of the desired values $u_i$ on the reconstruction, we recompute the last example with the same parameters $\alpha,\beta$ but 10\% increased values, i.e., $u_i=1.1(i-1)/10$, $i=1,\dots,5$. The results are shown in \cref{fig:transmission-ui}, where we repeat the exact coefficient from \cref{fig:transmission-exact} with adjusted labels in \cref{fig:transmission-ui:exact} for better comparison. As can be seen from \cref{fig:transmission-ui:mb}, the reconstruction is similar to that for $\alpha=0$. In particular, the total variation penalty prevents the misspecified desired values from being enforced strongly. This demonstrates that while misspecified values clearly do not have the same positive influence on the reconstruction, they at least do not have a negative influence.
\begin{figure}[t]
  \centering
  \begin{subfigure}{\linewidth}
    \centering
    \includegraphics[width=0.45\linewidth]{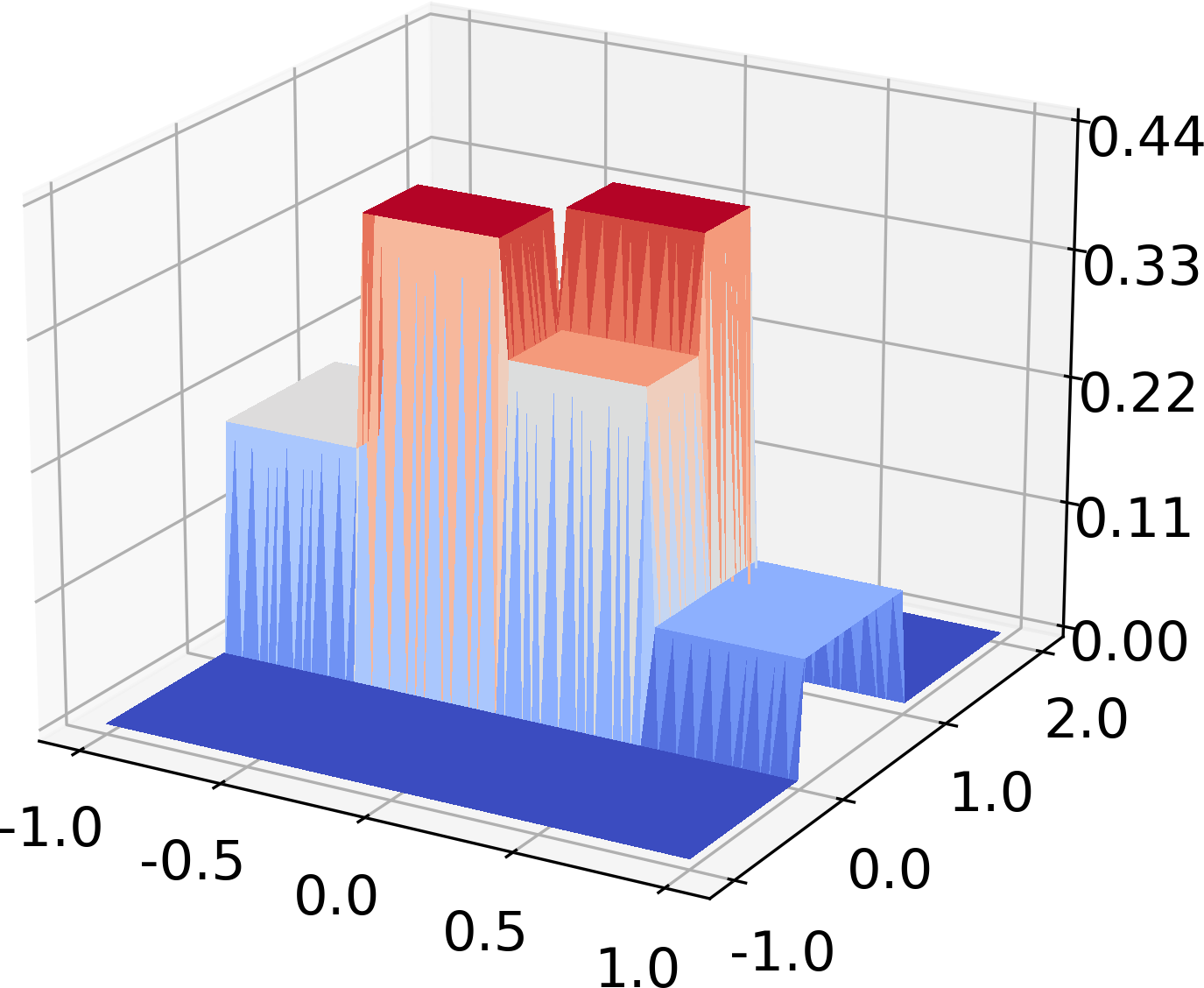}
    \hfill
    \includegraphics[width=0.45\linewidth]{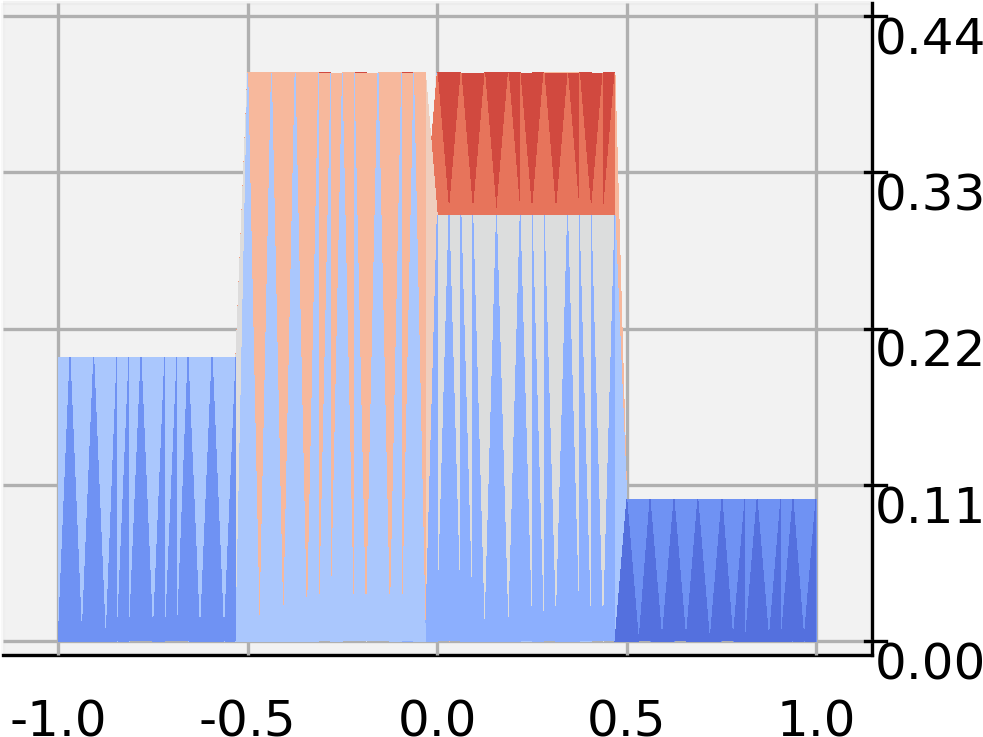}
    \caption{exact coefficient from \cref{fig:transmission-exact} for comparison}
    \label{fig:transmission-ui:exact}
  \end{subfigure}

  \begin{subfigure}{\linewidth}
    \centering
    \includegraphics[width=0.45\linewidth]{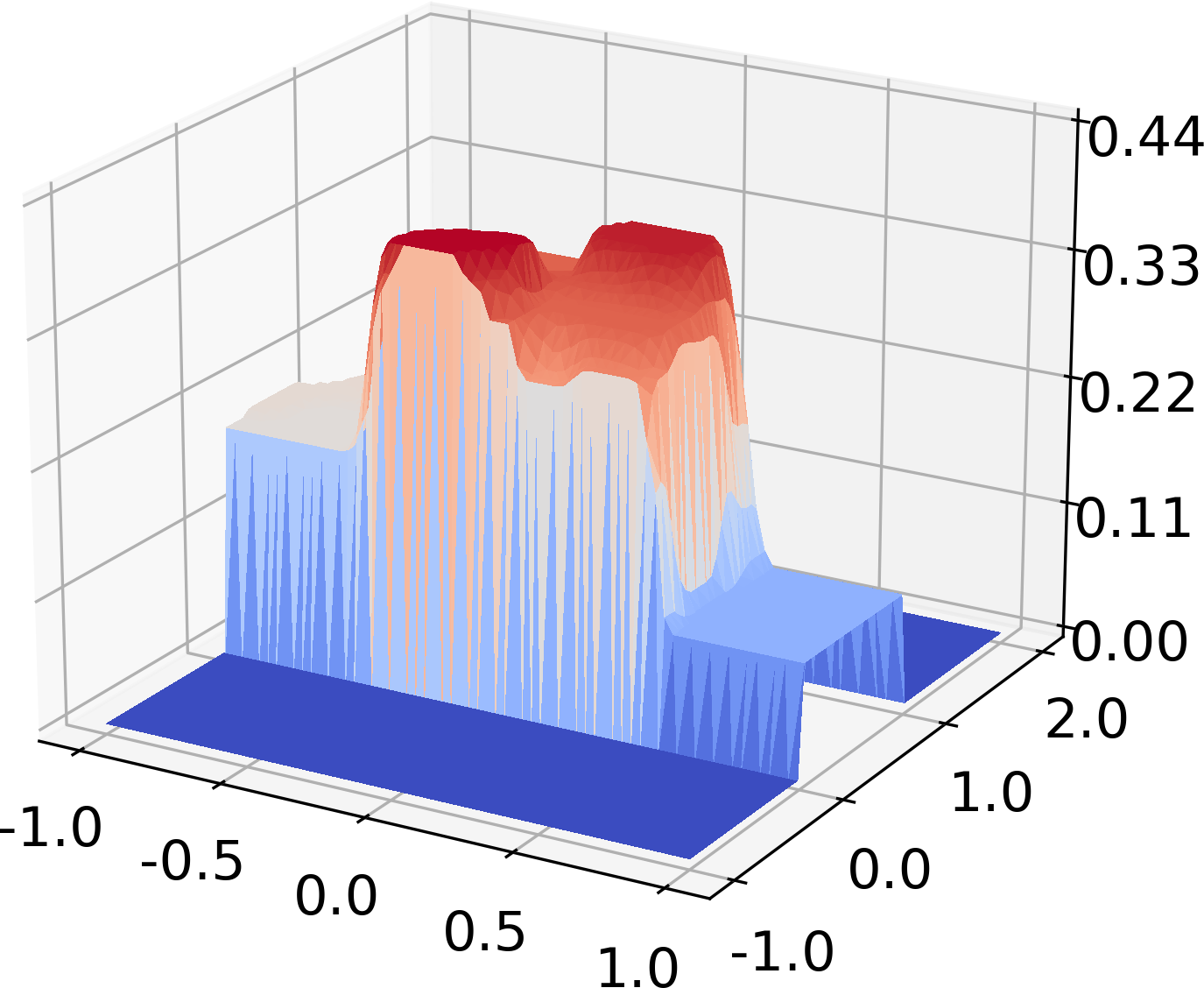}
    \hfill
    \includegraphics[width=0.45\linewidth]{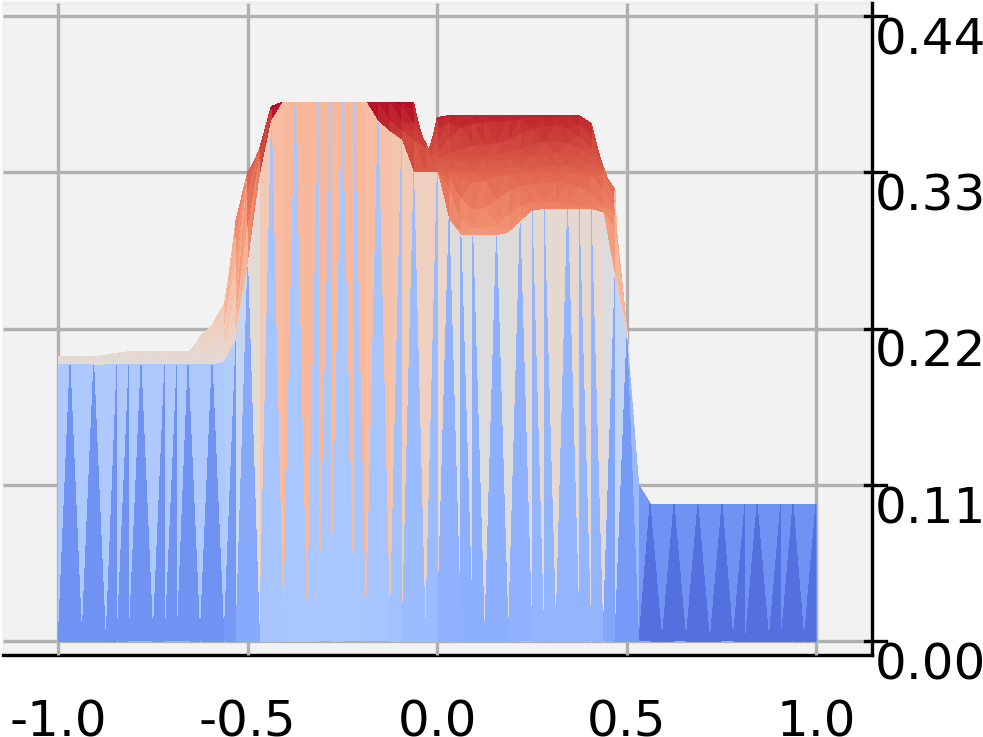}
    \caption{$u_i\in\{0,0.11,0.22,0.33,0.44\}$, $\alpha = 10^{-5}$, $\beta = 10^{-4}$, $820$ iterations}
    \label{fig:transmission-ui:mb}
  \end{subfigure}
  \caption{transmission example, effect of variation of $u_i$ on reconstruction}
  \label{fig:transmission-ui}
\end{figure}

\subsection{A model seismic inverse problem}\label{sec:ex-reflection}

We next consider an example which is inspired from seismic tomography.
We assume that the data is given in the form of a time series of mean values of the reflected waves $y$ over certain spatial regions $O_i$. Thus we define the observation space $\calO=L^2(I)^m$ for $m\in \mathbb N$ and the observation operator
\begin{equation*}
  B\colon L^2(Q)\to \calO,\qquad y\mapsto \left(\frac1{|O_i|}\int_{O_i}y(\cdot,x)\dx\right)_{i=1}^m,
\end{equation*}
where the $O_i\subset \Omega$ are the $m$ spatial observation patches. Furthermore we assume that seismic sources are given by $s$ point sources located on the surface $\Gamma_s\subset \partial\Omega$ whose magnitudes are time dependent and follow a Ricker wavelet of the form
\begin{equation*}
  f_k(t)=a_k(1-2 \pi^2 h_k^2 (t-t_k)^2) e^{-\pi^2 h_k^2 (t-t_k)^2}
\end{equation*}
with $h,a,t\in \mathbb R^s$. This leads to the modified state equation
\begin{equation*}
  \left\{
    \begin{aligned}
      \partial_{tt} y - \div(u\nabla y) &= 0 &&\text{ in } Q, \\
      {\partial_\nu y} &= \sum_{k=1}^sf_k\delta_{x_k} && \text{ on } (0,T) \times \Gamma_s, \\
      {\partial_\nu y} &= 0&& \text{ on } (0,T) \times \partial \Omega\setminus \Gamma_s, \\
      y(0) &= 0,\quad {\partial_t} y(0) = 0, && \text{ on }\Omega
    \end{aligned}
  \right.
\end{equation*}
with $(x_k)_{k=1}^s\subset \Gamma_s$, $(f_k)_{k=1}^s\subset L^2(I)$, and $\delta_{x_k}$ the Dirac measure supported on $x_k$.
In our concrete example, we chose $\Omega=(-1,1)^2$, $\Gamma_s=(-1,1)\times \{1\}$, and $T=3$. We set $\Omega_c=(-1,1)\times (-1,0.7)$. The observation patches are chosen as
\begin{equation*}
  O_i=(o_i,o_i+0.2)\times (0.8,1) \quad\text{ with }\quad o_i\in\{-1,-0.8,-0.6,-0.4,-0.2,0,0.2,0.4,0.6,0.8\}.
\end{equation*}
The sources are located at $x=(-1+k\cdot 0.1,1)$ with $k=0,\ldots,20$. The parameters of the Ricker wavelet are set to $a_k=2$, $h_k=5$ and $t_k=0.1$. The offset $\hat u$ has the constant value $1$. Finally, the exact velocity model is given by
\begin{equation*}
  u_e=\begin{cases}
    3&x\in (0.4,0.6)\times (0.1,0.4),\\
    2&x\in (-0.8,-0.5)\times (0.2,0.6),\\
    1&x\in (-0.2,0.2)\times (0.3,0.5),\\
    0&\text{else,}
  \end{cases}
\end{equation*}
for the constant reference coefficient $\hat u \equiv 1$, cf. \cref{fig:reflection-recon:exact}.

The recorded data for our experiments are generated by solving the state equation with the exact velocity model $u_e$ resulting in the exact state $y_e$. Then we set $y_d=By_e+\delta n$ with $\delta\in [0,1]$. The function $n\in L^\infty(I)^m$ is a disturbance which models measurement errors and exterior influences. In our case we use a function of the form
\begin{equation*}
  n_k(t):=\eta_k r_k :=\eta_k\sum_{i=1}^M\frac{m_{i,k}}{i}\cos(4\pi t-s_{i,k}\pi),
\end{equation*}
where $M\in \mathbb N$, $\eta_k=\frac{\|(By_e)_k\|_{L^\infty(I)}}{\|r_k\|_{L^\infty(I)}}$, and $m_{i,k},s_{i,k}$ are uniform random numbers in $[0,1]$. Here we take $M=10$.

For the discretization, we take a tensorial-based triangular mesh with $N_h=129^2$, $N_\tau=129$, and $\sigma=1/4$. The relative noise level is $\delta =0.05$. An appropriate regularization parameter is given by $\beta = 10^{-4}$; for simplicity, we set $\alpha=0$. The iteration is initialized with $u_0=0$ and the stepsizes are again chosen as $\gamma_F=10^{-1}$ and $\gamma_G=10^3$. The iteration is stopped if the absolute residuum is smaller than $10^{-4}$; in this experiment, this was reached after $1068$ iterations.
\begin{figure}[t]
  \centering
  \begin{subfigure}{0.485\linewidth}
    \centering
    \includegraphics[width=\linewidth]{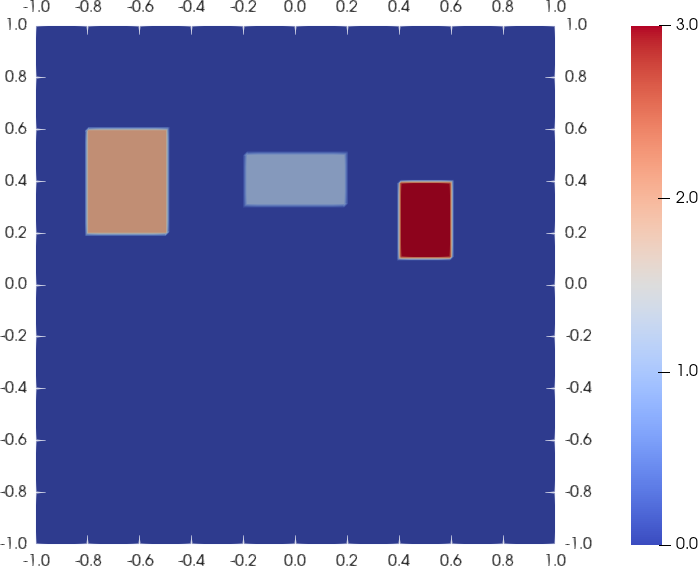}
    \caption{true coefficient $u_e$}\label{fig:reflection-recon:exact}
  \end{subfigure}
  \hfill
  \begin{subfigure}{0.485\linewidth}
    \centering
    \includegraphics[width=\linewidth]{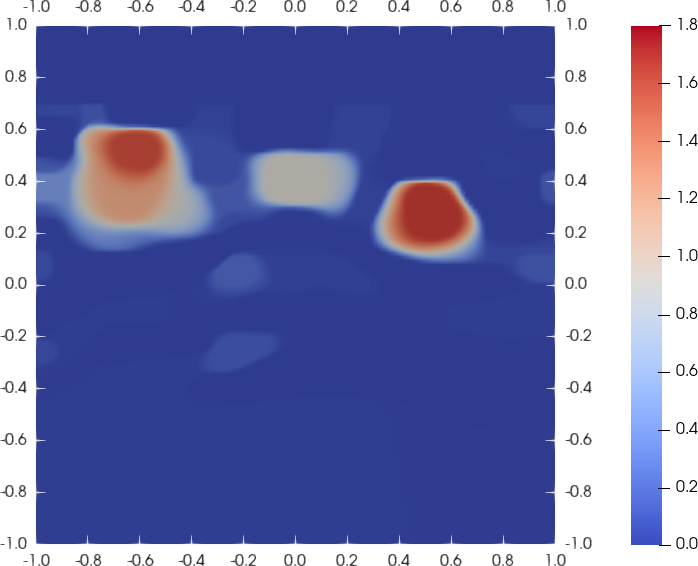}\quad
    \caption{reconstruction $\bar u$}\label{fig:reflection-recon:recon}
  \end{subfigure}
  \caption{reflection example: results for $\alpha=0$, $\beta=10^{-4}$, $1068$ iterations (note the different color bars)}\label{fig:reflection-recon}
\end{figure}
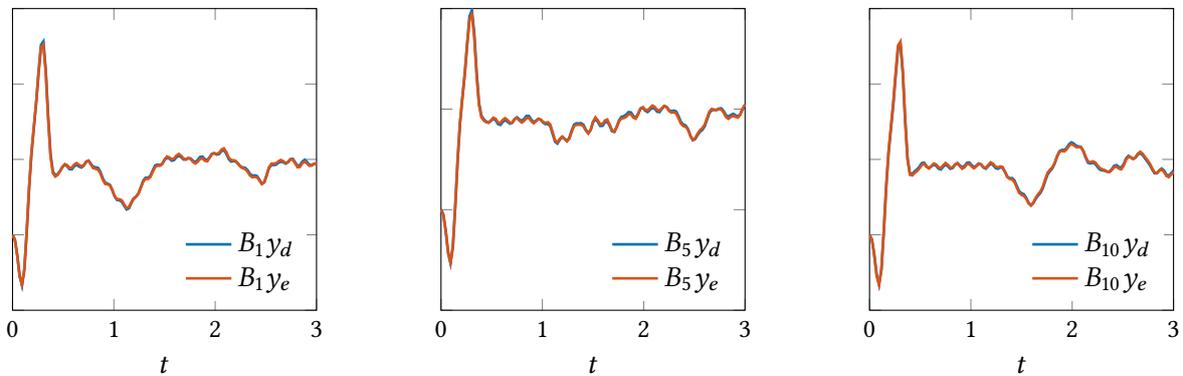
\begin{figure}[t]
  \centering
  \input{observation_1.tikz}\hfill
  \input{observation_5.tikz}\hfill
  \input{observation_10.tikz}
  \caption{reflection example: exact and noisy observations on $O_i$, $i=1,5,10$}\label{fig:reflection-obs}
\end{figure}
\Cref{fig:reflection-obs} shows the exact and noisy observations on $O_1$, $O_5$ and $O_{10}$. At the onset, we note two high spikes (a negative and a positive one) which are caused by the source wave initiated on boundary points $\Gamma_s$. The remaining oscillations are caused by the reflection waves originating from the discontinuities of $u_e$ and from the reflecting boundary; only these carry information about the coefficient, which makes the reconstruction challenging.
The results are shown in \cref{fig:reflection-recon:recon}, where each color map is scaled individually to show more details.
We observe that the positions of the discontinuities in $u_e$ that are close to the observation patches are well approximated in $\bar u$ and that the corresponding interfaces are quite sharp. However, the approximation quality of the discontinuities becomes worse farther away from the observation region. This is caused by the fact that reflected waves from lower sections of the discontinuities are more dispersed than the reflected waves from the upper sections of the discontinuities.

\section{Conclusion}

We showed existence of solutions to an optimal control problem for the wave equation with the control entering into the principal part of the operator using total variation regularization and a reformulation of pointwise constraints using a cutoff function. Preferential attainment of a discrete set of control values is incorporated through a multi-bang penalty. We also derived an improved regularity result for solutions of the wave equation under additional natural assumptions on the data and the control, which (for smooth cutoff functions) allows obtaining necessary optimality conditions that can be interpreted in a suitable pointwise fashion. Finally, we demonstrated that the optimal control problem can be solved numerically using a combination of a stabilized finite element discretization and a nonlinear primal-dual proximal splitting algorithm.

This work can be extended in several directions. Besides applying the proposed approach to more realistic models of acoustic tomography or seismic imaging for practical applications, it would be worthwhile to consider the case of boundary observations of the state \cite{Feng:2003}, which however may lead to an unbounded observation operator $B$.
A further challenging goal would be deriving sufficient second-order conditions. Such conditions could then be used for obtaining discretization error estimates for the optimal controls or for showing convergence of the nonlinear primal-dual proximal splitting algorithm based on the ``three-point condition'' on $S$ from \cite{CMV:2019}.

\section*{Acknowledgments}

Support by the German Science Fund (DFG) under grant CL 487/1-1 for C.C. and by the ERC advanced grant 668998 (OCLOC) under the EU's H2020 research program for K.K. and P.T. are gratefully acknowledged.

\bibliographystyle{jnsao}
\bibliography{tvwavecontrol}

\end{document}

%% file: observation_1.tikz
\definecolor{mycolor1}{rgb}{0.00000,0.44700,0.74100}%
\definecolor{mycolor2}{rgb}{0.85000,0.32500,0.09800}%
\begin{tikzpicture}

\begin{axis}[%
width=4cm,
height=4cm,
scale only axis,
xmin=0,
xmax=3,
ymin=-0.2,
ymax=0.6,
xlabel = $t$,
legend style={legend cell align=left,align=left,draw=none},
]
\addplot [color=mycolor1,solid,line width=1.0pt]
  table[row sep=crcr]{%
0	0\\
0.0234375	-0.0128443905537077\\
0.046875	-0.0546086705323512\\
0.0703125	-0.109580435287627\\
0.09375	-0.134358901584836\\
0.1171875	-0.0921623267243629\\
0.140625	0.00987444406195484\\
0.1640625	0.1242595558131\\
0.1875	0.217051163502349\\
0.2109375	0.28529568880858\\
0.234375	0.352969748269279\\
0.2578125	0.433686978081719\\
0.28125	0.503378470503722\\
0.3046875	0.512504818842485\\
0.328125	0.437830465870448\\
0.3515625	0.314833332703774\\
0.375	0.214149803202236\\
0.3984375	0.167948098740142\\
0.421875	0.157925601666357\\
0.4453125	0.160870834335953\\
0.46875	0.169036833994847\\
0.4921875	0.179575048508063\\
0.515625	0.185520689214959\\
0.5390625	0.181986246368869\\
0.5625	0.174143911322215\\
0.5859375	0.173795383850193\\
0.609375	0.182035791243712\\
0.6328125	0.185698905187831\\
0.65625	0.181153319292237\\
0.6796875	0.178519228345082\\
0.703125	0.183573709121986\\
0.7265625	0.193788279039367\\
0.75	0.196862920121318\\
0.7734375	0.18847229574804\\
0.796875	0.179914126344239\\
0.8203125	0.179064874473804\\
0.84375	0.177614635509075\\
0.8671875	0.167040908753229\\
0.890625	0.15010393327777\\
0.9140625	0.138127855660026\\
0.9375	0.135908290296077\\
0.9609375	0.132496992260272\\
0.984375	0.118674359625816\\
1.0078125	0.101637206758626\\
1.03125	0.0924408800016226\\
1.0546875	0.0914163701582283\\
1.078125	0.0879406884582042\\
1.1015625	0.0771285147479715\\
1.125	0.0673929797031509\\
1.1484375	0.0702755032131181\\
1.171875	0.0839446627358569\\
1.1953125	0.095381883154383\\
1.21875	0.100999335846943\\
1.2421875	0.109442318875429\\
1.265625	0.125937963420733\\
1.2890625	0.143318054263798\\
1.3125	0.151385326692317\\
1.3359375	0.151272179838881\\
1.359375	0.155964526347924\\
1.3828125	0.170323304505573\\
1.40625	0.183484434899826\\
1.4296875	0.186214752707998\\
1.453125	0.184526305364861\\
1.4765625	0.188036272079439\\
1.5	0.198540349148153\\
1.5234375	0.205322803229321\\
1.546875	0.202390553350616\\
1.5703125	0.196612109679766\\
1.59375	0.198081283514913\\
1.6171875	0.205769417316815\\
1.640625	0.209930375440289\\
1.6640625	0.205956071739833\\
1.6875	0.199749408894648\\
1.7109375	0.199339228032998\\
1.734375	0.203378989907328\\
1.7578125	0.202504095733702\\
1.78125	0.196058676545729\\
1.8046875	0.192932351961957\\
1.828125	0.19867061963687\\
1.8515625	0.206669521897156\\
1.875	0.207266951046439\\
1.8984375	0.20032167237875\\
1.921875	0.197961073100322\\
1.9453125	0.206099063009207\\
1.96875	0.215094258934523\\
1.9921875	0.216143903041982\\
2.015625	0.212558125606259\\
2.0390625	0.214374940014295\\
2.0625	0.222241596657854\\
2.0859375	0.224833062055885\\
2.109375	0.215324459206266\\
2.1328125	0.201267895255215\\
2.15625	0.194353454784183\\
2.1796875	0.195494681859143\\
2.203125	0.194554978022387\\
2.2265625	0.188259361651198\\
2.25	0.181111290960236\\
2.2734375	0.178620127471531\\
2.296875	0.177440167516876\\
2.3203125	0.17116174040625\\
2.34375	0.161471546822831\\
2.3671875	0.156785073965086\\
2.390625	0.157661319941333\\
2.4140625	0.155737241344287\\
2.4375	0.145700541218797\\
2.4609375	0.135607303514613\\
2.484375	0.139157699901233\\
2.5078125	0.157544723263598\\
2.53125	0.1767008804156\\
2.5546875	0.185312264114987\\
2.578125	0.185060162808382\\
2.6015625	0.186095144400315\\
2.625	0.191244684815823\\
2.6484375	0.19461905004786\\
2.671875	0.192272035441379\\
2.6953125	0.18882714778583\\
2.71875	0.192584949609827\\
2.7421875	0.200466218704255\\
2.765625	0.201507233185762\\
2.7890625	0.19223749529677\\
2.8125	0.18267968709235\\
2.8359375	0.184354054658922\\
2.859375	0.194231889294755\\
2.8828125	0.198931277702622\\
2.90625	0.193311205551893\\
2.9296875	0.185292463644856\\
2.953125	0.184917302126355\\
2.9765625	0.188791781038216\\
3	0.188718811301007\\
};
\addlegendentry{$B_1y_{d}$};

\addplot [color=mycolor2,solid,line width=1.0pt]
  table[row sep=crcr]{%
0	0\\
0.0234375	-0.0127093287466353\\
0.046875	-0.053835991238094\\
0.0703125	-0.107672114840398\\
0.09375	-0.131853949621265\\
0.1171875	-0.0904642511710364\\
0.140625	0.0100569806629279\\
0.1640625	0.126208111969495\\
0.1875	0.219154923714498\\
0.2109375	0.286639607702381\\
0.234375	0.352908508798019\\
0.2578125	0.431095939228179\\
0.28125	0.497378121305717\\
0.3046875	0.504574880628742\\
0.328125	0.42965807261719\\
0.3515625	0.309014237562856\\
0.375	0.209462342555179\\
0.3984375	0.163360266158729\\
0.421875	0.154529670588294\\
0.4453125	0.158905626294963\\
0.46875	0.168285993285294\\
0.4921875	0.180039919017887\\
0.515625	0.18751225534088\\
0.5390625	0.184806538562263\\
0.5625	0.177327881772237\\
0.5859375	0.177939131911115\\
0.609375	0.186555222642225\\
0.6328125	0.190257249085382\\
0.65625	0.185434935138224\\
0.6796875	0.181094958782008\\
0.703125	0.184974762032815\\
0.7265625	0.194122528171229\\
0.75	0.196061180567244\\
0.7734375	0.186336927878722\\
0.796875	0.176850063337623\\
0.8203125	0.175004810300811\\
0.84375	0.173087412890425\\
0.8671875	0.162751585078858\\
0.890625	0.146222482717591\\
0.9140625	0.134979563882307\\
0.9375	0.133934977176969\\
0.9609375	0.131544053675938\\
0.984375	0.118707389440032\\
1.0078125	0.102743224567359\\
1.03125	0.0945926956577547\\
1.0546875	0.0943124565775482\\
1.078125	0.0916071867764756\\
1.1015625	0.0807379418023346\\
1.125	0.0707966249796025\\
1.1484375	0.0735453817620109\\
1.171875	0.0864060734329328\\
1.1953125	0.0968969489721752\\
1.21875	0.101585663957135\\
1.2421875	0.109090085536648\\
1.265625	0.124429535494577\\
1.2890625	0.140925699418292\\
1.3125	0.148399571744242\\
1.3359375	0.148043936032772\\
1.359375	0.152808849045985\\
1.3828125	0.166776897502198\\
1.40625	0.179457098377273\\
1.4296875	0.183432543747814\\
1.453125	0.182975252814619\\
1.4765625	0.187674920092541\\
1.5	0.199451142051201\\
1.5234375	0.207729487643314\\
1.546875	0.205426778315106\\
1.5703125	0.200159627157059\\
1.59375	0.202110248778708\\
1.6171875	0.210108991134551\\
1.640625	0.214449481004982\\
1.6640625	0.209783893113506\\
1.6875	0.202014343345463\\
1.7109375	0.2004019437686\\
1.734375	0.203339549632029\\
1.7578125	0.201086006856777\\
1.78125	0.193188685993517\\
1.8046875	0.189119600110772\\
1.828125	0.193863140071962\\
1.8515625	0.201819429040284\\
1.875	0.20271970397424\\
1.8984375	0.196274313509809\\
1.921875	0.195024745488851\\
1.9453125	0.204315276059056\\
1.96875	0.214360628973128\\
1.9921875	0.216612927420113\\
2.015625	0.214561512770075\\
2.0390625	0.217442043398224\\
2.0625	0.226102187235773\\
2.0859375	0.229298490006299\\
2.109375	0.219643486556172\\
2.1328125	0.205301933798589\\
2.15625	0.198147604842099\\
2.1796875	0.197892081632301\\
2.203125	0.19592614727385\\
2.2265625	0.188589146009158\\
2.25	0.180279385314581\\
2.2734375	0.176389339565532\\
2.296875	0.174297198192002\\
2.3203125	0.167084272877654\\
2.34375	0.157165788413824\\
2.3671875	0.152756980553863\\
2.390625	0.153782309933906\\
2.4140625	0.152313539232083\\
2.4375	0.143475801614615\\
2.4609375	0.134531799316653\\
2.484375	0.139195870008161\\
2.5078125	0.158824496743192\\
2.53125	0.179264631703503\\
2.5546875	0.188578316820085\\
2.578125	0.188922009145499\\
2.6015625	0.189885212476212\\
2.625	0.195044059002417\\
2.6484375	0.198534266828105\\
2.671875	0.195195138385482\\
2.6953125	0.1905222704711\\
2.71875	0.193225585949276\\
2.7421875	0.200074412372887\\
2.765625	0.199819420558024\\
2.7890625	0.189512225769973\\
2.8125	0.179057511746324\\
2.8359375	0.179937059549983\\
2.859375	0.189677014658613\\
2.8828125	0.194412122413946\\
2.90625	0.18910230829366\\
2.9296875	0.182846520444534\\
2.953125	0.183626202369387\\
2.9765625	0.188484165151233\\
3	0.189529540197474\\
};
\addlegendentry{$B_1y_e$};

\end{axis}
\end{tikzpicture}%

%% file: observation_5.tikz
\definecolor{mycolor1}{rgb}{0.00000,0.44700,0.74100}%
\definecolor{mycolor2}{rgb}{0.85000,0.32500,0.09800}%
\begin{tikzpicture}

\begin{axis}[%
width=4cm,
height=4cm,
scale only axis,
xmin=0,
xmax=3,
ymin=-0.2,
ymax=0.4,
xlabel = $t$,
legend style={legend cell align=left,align=left,draw=none},
]
\addplot [color=mycolor1,solid,line width=1.0pt]
  table[row sep=crcr]{%
0	0\\
0.0234375	-0.0101815834103986\\
0.046875	-0.0429347995251187\\
0.0703125	-0.0863636877117446\\
0.09375	-0.106798693478095\\
0.1171875	-0.0740910170623699\\
0.140625	0.00692915441663963\\
0.1640625	0.0986272310762174\\
0.1875	0.171882349325649\\
0.2109375	0.222441701711244\\
0.234375	0.270561116687363\\
0.2578125	0.329923539078\\
0.28125	0.384354468899358\\
0.3046875	0.396552933090658\\
0.328125	0.350128992012129\\
0.3515625	0.272085717756738\\
0.375	0.212016532993628\\
0.3984375	0.187164717291593\\
0.421875	0.179823925749074\\
0.4453125	0.175082000792064\\
0.46875	0.173109908285348\\
0.4921875	0.176614881523936\\
0.515625	0.18049066081735\\
0.5390625	0.177725492032543\\
0.5625	0.17034580250318\\
0.5859375	0.168592816201045\\
0.609375	0.174700682556405\\
0.6328125	0.178301153402986\\
0.65625	0.174908249882117\\
0.6796875	0.17101908700582\\
0.703125	0.172858240830827\\
0.7265625	0.180400394511034\\
0.75	0.183644566279859\\
0.7734375	0.179438690912111\\
0.796875	0.176432306684986\\
0.8203125	0.180627206653578\\
0.84375	0.186585044389864\\
0.8671875	0.186234906953795\\
0.890625	0.179018726658775\\
0.9140625	0.174169679766593\\
0.9375	0.176971821280553\\
0.9609375	0.180766569273104\\
0.984375	0.178744657038451\\
1.0078125	0.173034327378343\\
1.03125	0.170198682555859\\
1.0546875	0.170966723252231\\
1.078125	0.165775240671729\\
1.1015625	0.15063624650425\\
1.125	0.135177328474976\\
1.1484375	0.131526452029419\\
1.171875	0.139080630665478\\
1.1953125	0.144552908702452\\
1.21875	0.140932092375293\\
1.2421875	0.136753621085003\\
1.265625	0.142930499826586\\
1.2890625	0.157484458585623\\
1.3125	0.169085971907497\\
1.3359375	0.171795488927892\\
1.359375	0.171481481240675\\
1.3828125	0.172042127122203\\
1.40625	0.168611718520511\\
1.4296875	0.15890850366256\\
1.453125	0.152372678740895\\
1.4765625	0.158608443173595\\
1.5	0.173636947498594\\
1.5234375	0.180803514305872\\
1.546875	0.174928439146645\\
1.5703125	0.166534755620614\\
1.59375	0.166779149546124\\
1.6171875	0.173751463981983\\
1.640625	0.175739380495911\\
1.6640625	0.166975626406432\\
1.6875	0.15604984391602\\
1.7109375	0.154815117427845\\
1.734375	0.165130487595078\\
1.7578125	0.17801327859298\\
1.78125	0.18469070526109\\
1.8046875	0.18599464588112\\
1.828125	0.189348526378687\\
1.8515625	0.19414670931218\\
1.875	0.193367998584213\\
1.8984375	0.18612335941283\\
1.921875	0.182930558671751\\
1.9453125	0.190701589753824\\
1.96875	0.201191402458868\\
1.9921875	0.20284311235523\\
2.015625	0.19669688157545\\
2.0390625	0.194238680607126\\
2.0625	0.198905575314856\\
2.0859375	0.203146920396053\\
2.109375	0.200680196255714\\
2.1328125	0.194862945932867\\
2.15625	0.194082442056189\\
2.1796875	0.200515890162327\\
2.203125	0.205959884053438\\
2.2265625	0.204144024296947\\
2.25	0.197845925152632\\
2.2734375	0.195575954454038\\
2.296875	0.197212364288986\\
2.3203125	0.195328443715898\\
2.34375	0.185833001202262\\
2.3671875	0.174702807785188\\
2.390625	0.16924633438253\\
2.4140625	0.167578704078098\\
2.4375	0.160417441124306\\
2.4609375	0.147047382193501\\
2.484375	0.138271784379911\\
2.5078125	0.140375672137391\\
2.53125	0.148258119272795\\
2.5546875	0.154845428355905\\
2.578125	0.158534986415399\\
2.6015625	0.165725180407972\\
2.625	0.179054810698592\\
2.6484375	0.192279430677565\\
2.671875	0.198417344887619\\
2.6953125	0.198164370240088\\
2.71875	0.198394260070468\\
2.7421875	0.201070190224048\\
2.765625	0.201722676307305\\
2.7890625	0.196001235216864\\
2.8125	0.187620489204941\\
2.8359375	0.184679449813329\\
2.859375	0.188193983374337\\
2.8828125	0.190671301604649\\
2.90625	0.187688457055233\\
2.9296875	0.184795366809877\\
2.953125	0.190082980745678\\
2.9765625	0.20044975004579\\
3	0.208016970093768\\
};
\addlegendentry{$B_5y_{d}$};

\addplot [color=mycolor2,solid,line width=1.0pt]
  table[row sep=crcr]{%
0	0\\
0.0234375	-0.0100272301978231\\
0.046875	-0.0421539011844964\\
0.0703125	-0.0845716954118823\\
0.09375	-0.10457676741022\\
0.1171875	-0.072663587687376\\
0.140625	0.00707355160368754\\
0.1640625	0.100043433763186\\
0.1875	0.173184076399722\\
0.2109375	0.222876789260861\\
0.234375	0.269547755000225\\
0.2578125	0.326414073492266\\
0.28125	0.377836302008936\\
0.3046875	0.388756740249954\\
0.328125	0.342606863268824\\
0.3515625	0.267015088563068\\
0.375	0.208149906385462\\
0.3984375	0.183622479247301\\
0.421875	0.177452621843198\\
0.4453125	0.173984812579266\\
0.46875	0.173087703673513\\
0.4921875	0.177682252781768\\
0.515625	0.182928695180758\\
0.5390625	0.180669929043778\\
0.5625	0.173401357231255\\
0.5859375	0.172335619842685\\
0.609375	0.17856936938259\\
0.6328125	0.181986558708106\\
0.65625	0.178123605628987\\
0.6796875	0.172727799520633\\
0.703125	0.173516339306279\\
0.7265625	0.180096972709791\\
0.75	0.182358822078642\\
0.7734375	0.176998320108519\\
0.796875	0.173335651015277\\
0.8203125	0.176814637467141\\
0.84375	0.18256933706384\\
0.8671875	0.182629233107504\\
0.890625	0.175948270982666\\
0.9140625	0.171881514111047\\
0.9375	0.175750809922923\\
0.9609375	0.180458059797193\\
0.984375	0.17929121871791\\
1.0078125	0.174532387290357\\
1.03125	0.172536195440735\\
1.0546875	0.173813964630384\\
1.078125	0.169150001470978\\
1.1015625	0.15378142041034\\
1.125	0.137984952129151\\
1.1484375	0.134051105662821\\
1.171875	0.140799379564019\\
1.1953125	0.145398779634294\\
1.21875	0.14094943185613\\
1.2421875	0.135944871693676\\
1.265625	0.14108391332824\\
1.2890625	0.154986796049942\\
1.3125	0.166220638264591\\
1.3359375	0.168879606375197\\
1.359375	0.168780184079772\\
1.3828125	0.169174869302675\\
1.40625	0.165587318364551\\
1.4296875	0.157062815464394\\
1.453125	0.151644122873539\\
1.4765625	0.158936468171406\\
1.5	0.175097579804501\\
1.5234375	0.183553954856072\\
1.546875	0.177996960658931\\
1.5703125	0.169866022134144\\
1.59375	0.170352895899896\\
1.6171875	0.1773993788997\\
1.640625	0.17931425878353\\
1.6640625	0.16975767140772\\
1.6875	0.157451299499502\\
1.7109375	0.155159166750669\\
1.734375	0.164477848650013\\
1.7578125	0.176092530046783\\
1.78125	0.181573040475935\\
1.8046875	0.182246199438741\\
1.828125	0.184923570420573\\
1.8515625	0.189920447009106\\
1.875	0.189617032368782\\
1.8984375	0.182998417883239\\
1.921875	0.180880185747815\\
1.9453125	0.189705690095861\\
1.96875	0.201169706820477\\
1.9921875	0.203920021149959\\
2.015625	0.199149387010755\\
2.0390625	0.197440793048551\\
2.0625	0.202610460764858\\
2.0859375	0.207180278541522\\
2.109375	0.204377334841778\\
2.1328125	0.198124452169111\\
2.15625	0.196931726679572\\
2.1796875	0.202106299988089\\
2.203125	0.206603945663816\\
2.2265625	0.203844655489745\\
2.25	0.196511803798102\\
2.2734375	0.193026534397181\\
2.296875	0.194035962968719\\
2.3203125	0.191499532054445\\
2.34375	0.182013735607522\\
2.3671875	0.171316728261733\\
2.390625	0.166177809327186\\
2.4140625	0.165090371311927\\
2.4375	0.159040856588293\\
2.4609375	0.146699192664912\\
2.484375	0.138903405181971\\
2.5078125	0.142109077631497\\
2.53125	0.151043116980202\\
2.5546875	0.158056396213475\\
2.578125	0.162089551698106\\
2.6015625	0.169027760854973\\
2.625	0.182188866092421\\
2.6484375	0.195302346092313\\
2.671875	0.20045848331749\\
2.6953125	0.19911076807214\\
2.71875	0.198413205608535\\
2.7421875	0.200170578809899\\
2.765625	0.199656490740149\\
2.7890625	0.193156003620263\\
2.8125	0.184144403236746\\
2.8359375	0.180689838201864\\
2.859375	0.184294956716789\\
2.8828125	0.18701758015452\\
2.90625	0.184527710628174\\
2.9296875	0.183172753784463\\
2.953125	0.189476529144065\\
2.9765625	0.200728994949132\\
3	0.209317130441471\\
};
\addlegendentry{$B_5y_{e}$};

\end{axis}
\end{tikzpicture}%

%% file: observation_10.tikz
\definecolor{mycolor1}{rgb}{0.00000,0.44700,0.74100}%
\definecolor{mycolor2}{rgb}{0.85000,0.32500,0.09800}%
\begin{tikzpicture}

\begin{axis}[%
width=4cm,
height=4cm,
scale only axis,
xmin=0,
xmax=3,
ymin=-0.2,
ymax=0.6,
xlabel = $t$,
legend style={legend cell align=left,align=left,draw=none},
]
\addplot [color=mycolor1,solid,line width=1.0pt]
  table[row sep=crcr]{%
0	0\\
0.0234375	-0.0124751026297318\\
0.046875	-0.0529202084362192\\
0.0703125	-0.106504008071194\\
0.09375	-0.131900721103567\\
0.1171875	-0.0918971153866922\\
0.140625	0.00889058265554464\\
0.1640625	0.124329606902471\\
0.1875	0.216424201906396\\
0.2109375	0.281336837411661\\
0.234375	0.344528807217533\\
0.2578125	0.422475264860174\\
0.28125	0.493714664452819\\
0.3046875	0.506553540636934\\
0.328125	0.430633959334886\\
0.3515625	0.304420354050377\\
0.375	0.20255878694359\\
0.3984375	0.162396457036843\\
0.421875	0.161669343860103\\
0.4453125	0.168506002695181\\
0.46875	0.173327293383182\\
0.4921875	0.180499533325655\\
0.515625	0.190280493995781\\
0.5390625	0.191071741081776\\
0.5625	0.1819363965896\\
0.5859375	0.176742566526283\\
0.609375	0.181729835891484\\
0.6328125	0.186833905032427\\
0.65625	0.183343921162618\\
0.6796875	0.175866182100507\\
0.703125	0.173228504073788\\
0.7265625	0.18039242642238\\
0.75	0.186587868789751\\
0.7734375	0.181677217650836\\
0.796875	0.175146635064575\\
0.8203125	0.177276907940529\\
0.84375	0.186182835039796\\
0.8671875	0.190333847133266\\
0.890625	0.184430924749789\\
0.9140625	0.178445607498415\\
0.9375	0.182334683460181\\
0.9609375	0.190773378302439\\
0.984375	0.191735264680903\\
1.0078125	0.184528476259013\\
1.03125	0.179400440922749\\
1.0546875	0.183588492464865\\
1.078125	0.189236326214602\\
1.1015625	0.186274777671176\\
1.125	0.178152626819883\\
1.1484375	0.175350285416026\\
1.171875	0.182437625652074\\
1.1953125	0.188986119652329\\
1.21875	0.185248552682048\\
1.2421875	0.177418819074679\\
1.265625	0.175076833957297\\
1.2890625	0.1788026128811\\
1.3125	0.178227910704111\\
1.3359375	0.167196475820601\\
1.359375	0.154241790374725\\
1.3828125	0.149839850731003\\
1.40625	0.150514988275578\\
1.4296875	0.144288389252095\\
1.453125	0.129687189531168\\
1.4765625	0.115732897295192\\
1.5	0.110195727402514\\
1.5234375	0.107626169304238\\
1.546875	0.098046069660376\\
1.5703125	0.0846039540378792\\
1.59375	0.0780828459085076\\
1.6171875	0.0844036111975032\\
1.640625	0.0958830604752613\\
1.6640625	0.104123813477131\\
1.6875	0.111477901273804\\
1.7109375	0.124090031624681\\
1.734375	0.143335690185159\\
1.7578125	0.160191703200213\\
1.78125	0.169108336816475\\
1.8046875	0.176343223844326\\
1.828125	0.18978714193746\\
1.8515625	0.208381352444204\\
1.875	0.221898227828109\\
1.8984375	0.225080984375356\\
1.921875	0.224730639432573\\
1.9453125	0.230167802111274\\
1.96875	0.240148906381161\\
1.9921875	0.245177954326297\\
2.015625	0.241835657714896\\
2.0390625	0.236066225844217\\
2.0625	0.234677078512765\\
2.0859375	0.23269110392544\\
2.109375	0.222331990651985\\
2.1328125	0.204871767101179\\
2.15625	0.191154013599873\\
2.1796875	0.189206217131161\\
2.203125	0.191204271570954\\
2.2265625	0.188273777327556\\
2.25	0.180032479163542\\
2.2734375	0.172771023964954\\
2.296875	0.170943326941336\\
2.3203125	0.169937602566272\\
2.34375	0.167468691987378\\
2.3671875	0.167499574443061\\
2.390625	0.174209447752466\\
2.4140625	0.183116455688293\\
2.4375	0.184763963432022\\
2.4609375	0.179477002941186\\
2.484375	0.17777166761386\\
2.5078125	0.187514934448752\\
2.53125	0.202507195871242\\
2.5546875	0.209518621510553\\
2.578125	0.205689274239457\\
2.6015625	0.201094031606826\\
2.625	0.204782850032016\\
2.6484375	0.213020711424944\\
2.671875	0.21665785449466\\
2.6953125	0.211446034133752\\
2.71875	0.202407069719959\\
2.7421875	0.194381410438454\\
2.765625	0.185093055759554\\
2.7890625	0.173360573356423\\
2.8125	0.16336925995595\\
2.8359375	0.163010850531413\\
2.859375	0.170931256091218\\
2.8828125	0.175837387428906\\
2.90625	0.170511112795107\\
2.9296875	0.159475485923089\\
2.953125	0.156483547157623\\
2.9765625	0.163412860594974\\
3	0.170298303495956\\
};
\addlegendentry{$B_{10}y_{d}$};

\addplot [color=mycolor2,solid,line width=1.0pt]
  table[row sep=crcr]{%
0	0\\
0.0234375	-0.0127151833297519\\
0.046875	-0.0536077620673733\\
0.0703125	-0.107284608504415\\
0.09375	-0.132050192365961\\
0.1171875	-0.0914693713802763\\
0.140625	0.00899730932658491\\
0.1640625	0.126309780318186\\
0.1875	0.220088315370696\\
0.2109375	0.286339978933994\\
0.234375	0.350972638504467\\
0.2578125	0.430763211830244\\
0.28125	0.502102612761807\\
0.3046875	0.512145273811173\\
0.328125	0.432955576639275\\
0.3515625	0.304156587127637\\
0.375	0.200882922239373\\
0.3984375	0.159129739861345\\
0.421875	0.15756797476308\\
0.4453125	0.164285382772434\\
0.46875	0.169048827787768\\
0.4921875	0.17607315638973\\
0.515625	0.185652462119637\\
0.5390625	0.187929953214199\\
0.5625	0.180192974673292\\
0.5859375	0.176043753348168\\
0.609375	0.182401027091562\\
0.6328125	0.188965332602783\\
0.65625	0.187000020035969\\
0.6796875	0.179580422338644\\
0.703125	0.177070986690552\\
0.7265625	0.184400979300989\\
0.75	0.190494211697032\\
0.7734375	0.185472965864036\\
0.796875	0.177873131894425\\
0.8203125	0.178937681348341\\
0.84375	0.186452975820852\\
0.8671875	0.189253369595211\\
0.890625	0.18216149333227\\
0.9140625	0.17524623155911\\
0.9375	0.178897769344569\\
0.9609375	0.187225778353135\\
0.984375	0.188259752141761\\
1.0078125	0.180990660940571\\
1.03125	0.176392396240931\\
1.0546875	0.181546340045782\\
1.078125	0.188194745654487\\
1.1015625	0.186438385146896\\
1.125	0.179369500819255\\
1.1484375	0.177678567503027\\
1.171875	0.1854103461037\\
1.1953125	0.192239982267112\\
1.21875	0.188589587092093\\
1.2421875	0.1807726939579\\
1.265625	0.178582141882679\\
1.2890625	0.18146768121911\\
1.3125	0.179862797300864\\
1.3359375	0.167740895819285\\
1.359375	0.153773133446766\\
1.3828125	0.148181592856525\\
1.40625	0.147076020126086\\
1.4296875	0.140276404021548\\
1.453125	0.125433322731474\\
1.4765625	0.111399307738207\\
1.5	0.105758040152654\\
1.5234375	0.103348139183336\\
1.546875	0.0953443439704924\\
1.5703125	0.0831528383808342\\
1.59375	0.077842436298711\\
1.6171875	0.0854967467104077\\
1.640625	0.0985253201121137\\
1.6640625	0.108013745055124\\
1.6875	0.11542273147311\\
1.7109375	0.128046309974329\\
1.734375	0.147485733853336\\
1.7578125	0.164727739159201\\
1.78125	0.173120325386615\\
1.8046875	0.179031755638239\\
1.828125	0.191152852984703\\
1.8515625	0.208161508286844\\
1.875	0.220272492394056\\
1.8984375	0.222199105732674\\
1.921875	0.221184347400259\\
1.9453125	0.226336814967566\\
1.96875	0.235968511645425\\
1.9921875	0.24071202532914\\
2.015625	0.237180155929783\\
2.0390625	0.232649491853746\\
2.0625	0.23256316497962\\
2.0859375	0.231938041757617\\
2.109375	0.222973419383728\\
2.1328125	0.206758035703212\\
2.15625	0.194393862351149\\
2.1796875	0.192663301994041\\
2.203125	0.194964796225419\\
2.2265625	0.192228785467899\\
2.25	0.184085801345191\\
2.2734375	0.176736387184046\\
2.296875	0.173740037022717\\
2.3203125	0.171605494834516\\
2.34375	0.167725617930198\\
2.3671875	0.166484900619769\\
2.390625	0.171941443292989\\
2.4140625	0.179637200608236\\
2.4375	0.180889140488244\\
2.4609375	0.175473115701353\\
2.484375	0.173755275502882\\
2.5078125	0.183421327470192\\
2.53125	0.198923302239321\\
2.5546875	0.207215590292951\\
2.578125	0.204592199122622\\
2.6015625	0.20126582714705\\
2.625	0.206141205349927\\
2.6484375	0.21580849970511\\
2.671875	0.220188174694562\\
2.6953125	0.215086599624898\\
2.71875	0.206057565094848\\
2.7421875	0.198112089136862\\
2.765625	0.18901522060282\\
2.7890625	0.176396506567609\\
2.8125	0.165352626316316\\
2.8359375	0.163755744850376\\
2.859375	0.170254801134021\\
2.8828125	0.173724283987506\\
2.90625	0.166917108784146\\
2.9296875	0.155948400741896\\
2.953125	0.152942618521515\\
2.9765625	0.159723712496379\\
3	0.166348161509885\\
};
\addlegendentry{$B_{10}y_{e}$};

\end{axis}
\end{tikzpicture}%